\documentclass[12pt,reqno]{amsart}
\usepackage[margin=1in]{geometry}
\usepackage{amsmath,amssymb,amsthm,graphicx,amsxtra, setspace}
\usepackage[utf8]{inputenc}
\usepackage{mathrsfs}
\usepackage{hyperref}
\usepackage{xcolor}
\usepackage{upgreek}
\usepackage{mathtools}
\allowdisplaybreaks
\newtheorem{thm}{Theorem}[section]

\newtheorem{lem}{Lemma}[section]
\newtheorem{rem}{Remark}[section]

\newtheorem{Def}{Definition}[section]
\newtheorem{definition}{Definition}[section]

\usepackage{latexsym}
\usepackage{hyperref}
\usepackage{graphicx}
\usepackage{epstopdf}

\allowdisplaybreaks

\let\originalleft\left
\let\originalright\right
\renewcommand{\left}{\mathopen{}\mathclose\bgroup\originalleft}
\renewcommand{\right}{\aftergroup\egroup\originalright}

\newcommand{\Addresses}{{
		\footnote{

			\noindent	 \textsuperscript{1,2} Department of Applied Mathematics and Scientific Computing, Indian Institute of Technology Roorkee, Roorkee, 247667, India.	
			
			\noindent  \textit{e-mail\textsuperscript{1}:} \texttt{g\_gupta@as.iitr.ac.in}
			
			\noindent  \textit{e-mail\textsuperscript{2}:} \texttt{jay.dabas@gmail.com.}

			\noindent \textsuperscript{*}Corresponding author.
			
			\noindent\textbf{Keywords:} Controllability; Integro-differential equations; Neutral systems; Impulsive systems; Resolvent family.
			
			\vspace{0.2cm}
			\noindent\textbf{AMS 2000 Mathematics Subject Classification:} 93C23; 93B05; 34K30; 34G20; 47H10; 45J05.

}}}

\begin{document}
	\title[]{Study on Control Problem of a Impulsive Neutral Integro-Differential Equations with Fading Memory  \Addresses}
	\author [Garima Gupta.  Jaydev Dabas]{Garima Gupta\textsuperscript{1}.  Jaydev Dabas\textsuperscript{2*}}
	\maketitle
	\begin{abstract}
		This article addresses control problems for semilinear impulsive neutral integro-differential equations with memory in a Banach space. It investigates the approximate controllability of linear and semilinear systems and proves the establishment of mild solutions in the semilinear setting. The approach involves constructing a resolvent family for the corresponding integro-differential equation of linear type without memory. The results for the linear system are established first, then extended to the semilinear scenario, followed by a detailed example to illustrate the theoretical findings.
		
	\end{abstract}
	\section{\textbf{Introduction}}\label{intro}\setcounter{equation}{0}
	\vspace{0.5cm}
	Understanding how heat moves through complex materials requires better models than those used in classical thermal theory. The traditional heat equation works well for materials like copper, aluminium, silicon carbide, and polystyrene, where heat flow depends directly on temperature and its gradient. However, it doesn’t work well for materials with fading memory materials that react gradually to changes in heat. This is because the classical model wrongly assumes that any change in heat source instantly affects the material. To handle such cases accurately, more advanced models are needed.
	To solve this problem, Gurtin and Pikin \cite{gurtin1968general} and Nunziato \cite{nunziato1971heat} developed a better model for heat flow in materials with memory. In their approach, internal energy and heat flow depend on the material’s past behaviour, not just the present. Building on their work, Lunardi \cite{lunardi1990linear} studied a version of the linear heat equation for such materials in 1990, described as follows:
	\begin{align}\label{eq1.1} \begin{split}
			\frac{\mathrm{d}}{\mathrm{~d} t}\left[x(t, \zeta)+\int_{-\infty}^{t} k_{1}(t-s) x(s, \zeta) \mathrm{d} s\right]&=c \Delta x(t, \zeta) 
			+\int_{-\infty}^{t} k_{2}(t-s) \Delta x(s, \zeta) \mathrm{d}s\\ &\qquad \qquad+h(t, \zeta, x(t, \zeta)), t \geq 0, \zeta \in \Omega, \\
			 x(0, \zeta)&=0, \zeta \in \partial \Omega.
		\end{split}		
	\end{align}
	
	Here, for $(t, \zeta) \in[0, \infty) \times \Omega$, $x(t, \zeta)$ represents the value at the spatial point $\zeta$ and at a given time $t$, where the domain $\Omega \subset \mathbb{R}^{n}$ is open, bounded, and possesses a boundary $\partial \Omega$ of class $C^{2}$. The constant $c$ corresponds to a physical parameter, while the functions $k_{1}, k_{2}: \mathbb{R} \to \mathbb{R}$ characterise the internal energy and the relaxation of heat flux, respectively.
	
	The authors studied this system by turning it into a type of equation called an abstract Volterra equation and found important results about how the solution behaves over time, its smoothness, and whether it stays positive. For more on similar models and their uses, see \cite{wu2012theory} and related references.
	
	Recent investigations into semilinear integro-differential equations have employed the resolvent operator theory as a key analytical tool to understand their solutions, stability, and control. Notably, Dos Santos et al. \cite{dos2011existence} developed a resolvent operator framework for a type of linear neutral integro-differential equation
	\begin{align} \label{eq1.2} \begin{split}
			\frac{\mathrm{d}}{\mathrm{~d} t}\left[w(t)+\int_{0}^{t} \mathrm{G}(t-s) w(s) \mathrm{d} s\right] & =\mathrm{A} w(t)+\int_{0}^{t} \mathrm{~N}(t-s) w(s) \mathrm{d} s, t \in(0, T], \\
			w(0) & =\zeta \in \mathbb{W} .
		\end{split}		
	\end{align}
	
	The authors also explored whether mild, strict, and classical solutions exist for the problem \eqref{eq1.2}. In 2025, Arora and Nandakumaran \cite{arora2025controllability} studied control-related issues governed by a semilinear neutral integro-differential equation incorporating memory effects.
	
	This study is motivated by above discussed work. We investigate the existence of solutions and approximate controllability for a specific type of abstract impulsive neutral integro-differential equation defined on the interval  $J=[0,b]$. Our analysis takes place in a state space $\mathbb{X}$, which is a separable reflexive Banach space along with its dual space $\mathbb{X}^{*}$. Additionally, we employ a Hilbert space $\mathbb{U}$, which is identified with its dual space. The problem is presented in the following abstract form:
	
	\begin{align}\label{P}
		\begin{split}
			\frac{\mathrm{d}}{\mathrm{~d} t}\left[x(t)+\int_{-\infty}^{t} \mathrm{G}(t-s) x(s) \mathrm{d} s\right]= & \mathrm{A} x(t)+\int_{-\infty}^{t} \mathrm{~N}(t-s) x(s) \mathrm{d} s+\mathrm{B} u(t) \\
			& +f\left(t, x_{t}\right), \quad t \in [0, b] \backslash\left\{t_{1}, \ldots, t_{m}\right\}, \\ 
			\Delta x\left(t_{k}\right)&=\mathrm{D}_{k} x\left(t_{k}\right)+\mathrm{E}_{k} v_{k}, \quad k=1, \ldots, m, \\ 
			x_{0}= & \psi \in \mathfrak{B}. 
		\end{split}			
	\end{align}
	
Here, $\mathrm{A}: D(\mathrm{A}) \subseteq \mathbb{X} \to \mathbb{X}$ and $\mathrm{N}(t): D(\mathrm{N}(t)) \subseteq \mathbb{X} \to \mathbb{X}$ for all $t \geq 0$ be closed linear operators. The operators $\mathrm{G}(t): \mathbb{X} \to \mathbb{X}$ for $t \geq 0$, and $\mathrm{B}: \mathbb{U} \to \mathbb{X}$ are assumed to be bounded, with $\|\mathrm{B}\|_{\mathcal{L}(\mathbb{U}, \mathbb{X})} = M$. 

A nonlinear mapping $f: J \times \mathfrak{B} \to \mathbb{X}$ is given, where $\mathfrak{B}$ stands for the phase space, the structure of which will be discussed in detail later. Define the history function $x_t: (-\infty, 0] \to \mathbb{X}$ by $x_t(\theta) = x(t + \theta)$, and assume that $x_t \in \mathfrak{B}$ for every $t \geq 0$. 

The control input consists of a measurable function $u \in L^2(J, \mathbb{U})$, and a sequence $\{v_k\}_{k=1}^m$ with each $v_k \in \mathbb{U}$. At a discrete set of time instants $t_k$ (with $k = 1, \dots, m$, and $0 = t_0 < t_1 < \cdots < t_m < t_{m+1} = b$), the state variable may exhibit discontinuities, characterized by the jump condition
\[
\Delta x(t_k) = x(t_k^+) - x(t_k^-),
\]
where the one-sided limits are given by $x(t_k^\pm) = \lim_{h \to 0^\pm} x(t_k + h)$, and we assume $x(t_k^-) = x(t_k)$.

The notation $\prod_{j=1}^{k} A_j$ denotes the composition $A_1 A_2 \cdots A_k$, with the convention that $\prod_{j=k+1}^{k} A_j = \mathrm{I}$ (i.e., the identity operator). Similarly, $\prod_{j=k}^{1} A_j$ stands for $A_k A_{k-1} \cdots A_1$, and again $\prod_{j=k+1}^{k} A_j = \mathrm{I}$.

Impulsive systems, which involve sudden state changes, are common in fields like AI, biology, and population dynamics. These systems are modelled using differential equations with jumps at specific times. For a comprehensive introduction to the theory of impulsive systems, readers may consult the monograph by Lakshmikantham et al. \cite{lakshmikantham1989theory}. Controllability of such systems especially semilinear impulsive ones is key to managing real-world processes with abrupt shifts. Prior work has focused on both finite and infinite-dimensional spaces \cite{george2000note,arora2020approximate,arora2021approximate}, with recent studies using impulsive resolvent operators in Hilbert spaces \cite{mahmudov2024study,asadzade2024approximate,asadzade2025remarks,gupta2024existence}. However, no research has addressed the approximate controllability of impulsive neutral integro-differential equations with fading memory in Banach spaces using this method. This study aims to fill that gap. We first outline key properties of the resolvent operator, build the impulsive resolvent for the linear system, and then analyse the approximate controllability and mild solutions of both the linear and semilinear systems.

\section{\textbf{Preliminaries}}\label{pre}\setcounter{equation}{0}

The symbol $\langle \cdot, \cdot \rangle$ denotes the duality pairing between the Banach space $\mathbb{X}$ and its topological dual $\mathbb{X}^*$. The spaces $\mathcal{L}(\mathbb{U}; \mathbb{X})$ and $\mathcal{L}(\mathbb{X})$ denote, respectively, the collections of all bounded linear operators from $\mathbb{U}$ to $\mathbb{X}$ and from $\mathbb{X}$ to itself. These operator spaces are equipped with the operator norms $\|\cdot\|_{\mathcal{L}(\mathbb{U}; \mathbb{X})}$ and $\|\cdot\|_{\mathcal{L}(\mathbb{X})}$. 

We now define the resolvent operator associated with the following abstract integro-differential equation as in \cite{dos2011existence}

\begin{align}\label{eq2.1}
	\begin{split}
		\frac{\mathrm{d}}{\mathrm{d}t} \left[x(t) + \int_0^t \mathrm{G}(t - s)x(s)\,\mathrm{d}s \right] &= \mathrm{A}x(t) + \int_0^t \mathrm{N}(t - s)x(s)\,\mathrm{d}s,\quad t \in (0, T]  \\
		x(0) &= \zeta \in \mathbb{X}.
	\end{split}
\end{align}

\begin{definition}[\cite{dos2011existence}]\label{def2.1}
	A collection of bounded linear operators $\left(\mathscr{R}(t)\right)_{t \geq 0}$ on the Banach space $\mathbb{X}$ is referred to as a \textit{resolvent operator} associated with equation \eqref{eq2.1} if the following conditions are satisfied:
	\begin{enumerate}
		\item The mapping $\mathscr{R}: [0, \infty) \to \mathcal{L}(\mathbb{X})$ is strongly continuous and exhibits exponential boundedness. In addition, it holds that $\mathscr{R}(0)z = z$ for every $z \in \mathbb{X}$.
		
		\item For each $z \in D(\mathrm{A})$, the function $t \mapsto \mathscr{R}(t)z$ belongs to the space $C([0, \infty); D_1(\mathrm{A})) \cap C^1((0, \infty); \mathbb{X})$, and it fulfills the following integral-differential identities:
	\end{enumerate}
	
	\begin{align}\label{eq2.2}
		\frac{\mathrm{d}}{\mathrm{d}t} \left[\mathscr{R}(t)z + \int_0^t \mathrm{G}(t - s)\mathscr{R}(s)z\,\mathrm{d}s \right] 
		= \mathrm{A}\mathscr{R}(t)z + \int_0^t \mathrm{N}(t - s)\mathscr{R}(s)z\,\mathrm{d}s,
	\end{align}
	
	\begin{align}\label{eq2.3}
		\frac{\mathrm{d}}{\mathrm{d}t} \left[\mathscr{R}(t)z + \int_0^t \mathscr{R}(t - s)\mathrm{G}(s)z\,\mathrm{d}s \right]
		= \mathscr{R}(t)\mathrm{A}z + \int_0^t \mathscr{R}(t - s)\mathrm{N}(s)z\,\mathrm{d}s,
	\end{align}
	
	for all $t \geq 0$.
\end{definition}

For the analysis, we assume that the conditions (Cd1)–(Cd4) in \cite{arora2025controllability} are satisfied. Furthermore, we suppose that $\sup_{t \in [0, T]} \|\mathscr{R}(t)\|_{\mathcal{L}(\mathbb{X})} \leq M_{\mathscr{R}}$ for some constant $M_{\mathscr{R}}$.

\subsection{Geometry of Banach Spaces and Duality Mapping}

In this subsection, we begin by introducing certain geometric characteristics of Banach spaces, followed by a discussion on the concept of duality mapping and its key properties.

\begin{definition}
	A Banach space $\mathbb{X}$ is called \emph{strictly convex} if, for any two elements $z_1, z_2 \in \mathbb{X}$ with $\|z_1\|_{\mathbb{X}} = \|z_2\|_{\mathbb{X}} = 1$ and $z_1 \neq z_2$, it holds that
	\[
	\|v z_1 + (1 - v) z_2\|_{\mathbb{X}} < 1, \quad \text{for all } 0 < v < 1.
	\]
\end{definition}

\begin{definition}
	A Banach space $\mathbb{X}$ is said to be \emph{uniformly convex} if for every $\epsilon > 0$, there exists a $\delta = \delta(\epsilon) > 0$ such that for all $z_1, z_2 \in \mathbb{X}$ with $\|z_1\|_{\mathbb{X}} = \|z_2\|_{\mathbb{X}} = 1$ and $\|z_1 - z_2\|_{\mathbb{X}} \geq \epsilon$, we have
	\[
	\left\| \frac{z_1 + z_2}{2} \right\|_{\mathbb{X}} \leq 1 - \delta.
	\]
\end{definition}

It is well known that every uniformly convex Banach space is also strictly convex. The modulus of convexity $\delta_{\mathbb{X}}$ of $\mathbb{X}$ is defined by
\[
\delta_{\mathbb{X}}(\epsilon) = \inf \left\{ 1 - \frac{\|z + y\|_{\mathbb{X}}}{2} : \|z\|_{\mathbb{X}} \leq 1, \|y\|_{\mathbb{X}} \leq 1, \|z - y\|_{\mathbb{X}} \geq \epsilon \right\}.
\]
This function is non-decreasing on the interval $[0, 2]$. A Banach space $\mathbb{X}$ is uniformly convex if and only if $\delta_{\mathbb{X}}(\epsilon) > 0$ for all $\epsilon > 0$ (see \cite{pisier1975martingales}).

\begin{definition}
	A Banach space $\mathbb{X}$ is said to be \emph{uniformly smooth} if for every $\epsilon > 0$, there exists a $\delta > 0$ such that for all $z_1, z_2 \in \mathbb{X}$ with $\|z_1\|_{\mathbb{X}} = 1$ and $\|z_2\|_{\mathbb{X}} \leq \delta$, the following inequality holds:
	\[
	\frac{1}{2}\left( \|z_1 + z_2\|_{\mathbb{X}} + \|z_1 - z_2\|_{\mathbb{X}} \right) - 1 \leq \epsilon \|z_2\|_{\mathbb{X}}.
	\]
\end{definition}

\begin{rem}
	A Banach space $\mathbb{X}$ is uniformly smooth if and only if its dual $\mathbb{X}^*$ is uniformly convex. Likewise, $\mathbb{X}$ is uniformly convex if and only if $\mathbb{X}^*$ is uniformly smooth \cite{pisier1975martingales}.
\end{rem}

\subsubsection*{Duality Mapping}

We now recall the concept of the duality mapping.

\begin{definition}[\cite{barbu2012convexity}]
	The duality mapping $\mathscr{J}: \mathbb{X} \rightarrow 2^{\mathbb{X}^*}$ is defined by
	\[
	\mathscr{J}[z] = \left\{ z^* \in \mathbb{X}^* : \langle z, z^* \rangle = \|z\|_{\mathbb{X}}^2 = \|z^*\|_{\mathbb{X}^*}^2 \right\}, \quad \forall z \in \mathbb{X}.
	\]
\end{definition}

It follows immediately that for any scalar $\alpha \in \mathbb{R}$ and $z \in \mathbb{X}$, we have $\mathscr{J}[\alpha z] = \alpha \mathscr{J}[z]$.

\begin{rem}
	\begin{enumerate}
		\item[(i)] If $\mathbb{X}$ is a reflexive Banach space, it is always possible to define an equivalent norm under which both $\mathbb{X}$ and its dual $\mathbb{X}^*$ become strictly convex (see \cite{barbu2012convexity}). Moreover, according to Milman's theorem, any uniformly convex Banach space is reflexive.
		
		\item[(ii)] When $\mathbb{X}^*$ is strictly convex, the duality mapping $\mathscr{J}$ becomes single-valued, strictly monotone, bijective, and demicontinuous. That is, if $z_k \to z$ in $\mathbb{X}$, then $\mathscr{J}[z_k] \rightharpoonup \mathscr{J}[z]$ weakly in $\mathbb{X}^*$ as $k \to \infty$.
	\end{enumerate}
\end{rem}

Let us define the space of piecewise continuous functions:
\[
\mathcal{PC}(J; \mathbb{X}) := \left\{ \psi: J \to \mathbb{X} \;\middle|\; \psi \text{ is piecewise continuous with jumps at } t_k \text{ satisfying } \psi(t_k^-) = \psi(t_k) \right\}.
\]
For $z \in \mathcal{PC}(J; \mathbb{X})$, the norm is defined by $\|z\|_{\mathcal{PC}} := \sup_{t \in J} \|z(t)\|$.

\subsection*{Phase Space}

We now introduce the concept of phase space $\mathfrak{B}$, following the functional framework of Hale and Kato \cite{hino2006functional}, suitably adapted for impulsive systems.

The phase space $\mathfrak{B}$ consists of $\mathbb{X}$-valued functions defined on the interval $(-\infty, 0]$, and is equipped with a seminorm $\|\cdot\|_{\mathfrak{B}}$. The space satisfies the following assumptions:

\begin{enumerate}
	\item[(A1)] Let $z: (-\infty, \mu + \vartheta) \to \mathbb{X}$ be piecewise continuous on $[\mu, \mu + \vartheta)$ for some $\vartheta > 0$, and assume $z_{\mu} \in \mathfrak{B}$. Then for each $t \in [\mu, \mu + \vartheta)$:
	\begin{enumerate}
		\item[(i)] The shifted function $z_t$ also belongs to $\mathfrak{B}$.
		\item[(ii)] There exists a constant $K_1 > 0$, independent of $z$, such that $\|z(t)\|_{\mathbb{X}} \leq K_1 \|z_t\|_{\mathfrak{B}}$.
		\item[(iii)] There exist functions $\Lambda: [0, \infty) \to [1, \infty)$ (continuous) and $Y: [0, \infty) \to [1, \infty)$ (locally bounded), independent of $z$, such that
		\[
		\|z_t\|_{\mathfrak{B}} \leq \Lambda(t - \mu) \sup_{\mu \leq s \leq t} \|z(s)\|_{\mathbb{X}} + Y(t - \mu)\|z_\mu\|_{\mathfrak{B}}.
		\]
	\end{enumerate}
	
	\item[(A2)] The map $t \mapsto z_t$ from $[\mu, \mu + \vartheta)$ into $\mathfrak{B}$ is continuous.
	
	\item[(A3)] The space $\mathfrak{B}$ is complete with respect to the seminorm $\|\cdot\|_{\mathfrak{B}}$.
\end{enumerate}

Let $\psi \in \mathfrak{B}$ be given. We define two auxiliary functions $f_1, f_2: J \to \mathbb{X}$ as
\[
f_1(t) = -\int_{-\infty}^0 \mathrm{G}(t - s)\psi(s)\,\mathrm{d}s, \qquad
f_2(t) = \int_{-\infty}^0 \mathrm{N}(t - s)\psi(s)\,\mathrm{d}s.
\]

In the next section, we formally define the impulsive system considered in \eqref{P}, inspired by the notion of mild solutions introduced in \cite{asadzade2025remarks}.

\begin{definition}
	Let $u \in L^2(J; \mathbb{U})$ be a given control. A function $x: (-\infty, b] \rightarrow \mathbb{X}$ is said to be a \emph{mild solution} of system \eqref{P} if the following conditions hold:
	
	\begin{enumerate}
		\item [(i)] $x(t) = \psi(t)$ for all $t < 0$, where $\psi \in \mathfrak{B}$ denotes the initial history;
		\item [(ii)]the function $f_1$ is differentiable and its derivative $f_1'$ belongs to $L^1(J; \mathbb{X})$;
		\item [(iii)] at each impulsive instant $t_k$ $(k = 1, \ldots, m)$, the jump condition
		\[
		\Delta x(t_k) = \mathrm{D}_k x(t_k) + \mathrm{E}_k v_k
		\]
		is satisfied;
		\item [(iv)] $x$ is continuous on each open interval $(t_{k-1}, t_k]$, and is given by the following expression:
	\end{enumerate}
	
	\begin{align}\label{2}
		x(t) = 
		\begin{cases}
			\mathscr{R}(t)\psi(0) + \displaystyle\int_0^t \mathscr{R}(t - s)\big[\mathrm{B}u(s) + f(s, x_s)\big]\,\mathrm{d}s \\
			\quad + \displaystyle\int_0^t \mathscr{R}(t - s)\big[f_1'(s) + f_2(s)\big]\,\mathrm{d}s, & 0 \leq t \leq t_1, \\[1.5ex]
			\mathscr{R}(t - t_k)x(t_k^+) + \displaystyle\int_{t_k}^t \mathscr{R}(t - s)\big[\mathrm{B}u(s) + f(s, x_s)\big]\,\mathrm{d}s \\
			\quad + \displaystyle\int_{t_k}^t \mathscr{R}(t - s)\big[f_1'(s) + f_2(s)\big]\,\mathrm{d}s, & t_k < t \leq t_{k+1},\quad k = 1, \ldots, m.
		\end{cases}
	\end{align}
	
	The post-impulse state $x(t_k^+)$ is given explicitly by:
	\begin{align*}
		x(t_k^+) &= \prod_{j=k}^{1} (\mathrm{I} + \mathrm{D}_j)\,\mathscr{R}(t_j - t_{j-1})\psi(0) \\
		&\quad + \sum_{i=1}^k \left[\prod_{j=k}^{i+1} (\mathrm{I} + \mathrm{D}_j)\,\mathscr{R}(t_j - t_{j-1})(\mathrm{I} + \mathrm{D}_i)\right] \\
		&\qquad \times \left( \int_{t_{i-1}}^{t_i} \mathscr{R}(t_i - s)[\mathrm{B}u(s) + f(s, x_s)]\,\mathrm{d}s + \int_{t_{i-1}}^{t_i} \mathscr{R}(t_i - s)[f_1'(s) + f_2(s)]\,\mathrm{d}s \right) \\
		&\quad + \sum_{i=2}^{k} \left[\prod_{j=k}^{i} (\mathrm{I} + \mathrm{D}_j)\,\mathscr{R}(t_j - t_{j-1})\right] \mathrm{E}_{i-1}v_{i-1} + \mathrm{E}_k v_k.
	\end{align*}
\end{definition}

	\begin{Def} [\cite{mahmudov2003approximate}]
	We say that the system \eqref{P} is approximately controllable on the interval $J$ if, for any initial function $\psi \in \mathfrak{B}$, any target state $h \in \mathbb{X}$, and for every tolerance $\epsilon > 0$, there exists a control input $u \in \mathrm{L}^2(J; \mathbb{U})$ such that the corresponding mild solution $x(\cdot)$ of system \eqref{P} fulfills the condition:
	\begin{equation*}
		\left\|x(b) - h\right\|_{\mathbb{X}} \leq \epsilon.
	\end{equation*}
	\end{Def}

\subsection{Controllability of Linear system}
Now, we proceed to establish some results to study the approximate controllability of the following neutral integro-differential linear system subject to impulsive effects:
	\begin{align}\label{LP}
		\begin{split}
			\frac{\mathrm{d}}{\mathrm{~d} t}\left[x(t)+\int_{-\infty}^{t} \mathrm{G}(t-s) x(s) \mathrm{d} s\right]&=\mathrm{A} x(t)+\int_{-\infty}^{t} \mathrm{~N}(t-s) x(s) \mathrm{d} s+\mathrm{B} u(t),\\ & \qquad \quad t \in J=[0, b] \backslash\left\{t_{1}, \ldots, t_{m}\right\}, \\ 
			\Delta x\left(t_{k}\right)&=\mathrm{D}_{k} x\left(t_{k}\right)+\mathrm{E}_{k} v_{k}, \quad k=1, \ldots, m, \\ 
			x(0)&= x_{0}   \in \mathbb{X}. 
		\end{split}			
	\end{align}
	
	We introduce a bounded linear operator  $\mathrm{M}: L^{2}(J, \mathbb{U}) \times \mathbb{U}^{m} \rightarrow \mathbb{X}$ defined as
	\begin{align*}
		& \mathrm{M}\left(u(\cdot),\left\{v_{k}\right\}_{k=1}^{m}\right) \\
		& =\mathscr{R}\left(b-t_{m}\right) \sum_{i=1}^{m} \prod_{j=m}^{i+1}\left(\mathrm{I}+\mathrm{D}_{j}\right) \mathscr{R}\left(t_{j}-t_{j-1}\right)\left(\mathrm{I}+\mathrm{D}_{i}\right) \int_{t_{i-1}}^{t_{i}} \mathscr{R}\left(t_{i}-s\right) \mathrm{B} u(s) d s \\
		&\qquad +\int_{t_{m}}^{b} \mathscr{R}(b-s) \mathrm{B} u(s) d s +\mathscr{R}\left(b-t_{m}\right) \sum_{i=2}^{m} \prod_{j=m}^{i}\left(\mathrm{I}+\mathrm{D}_{j}\right) \mathscr{R}\left(t_{j}-t_{j-1}\right) \mathrm{E}_{i-1} v_{i-1}\\& \qquad+\mathscr{R}\left(b-t_{m}\right) \mathrm{E}_{m} v_{m}.
	\end{align*}
The adjoint operator $M^{*}$ has the following form

$$
\begin{aligned}
	\mathrm{M}^{*} \varphi & =\left(\mathrm{B}^{*} \psi(\cdot),\left\{\mathrm{E}_{k}^{*} \psi\left(t_{k}^{+}\right)\right\}_{k=1}^{m}\right), \\
	\mathrm{B}^{*} \psi(t) & = \begin{cases}\mathrm{B}^{*} \mathscr{R}^{*}(b-t) \varphi, & t_{m}<t \leq b \\
		\mathrm{B}^{*} \mathscr{R}^{*}\left(t_{k}-t\right)\left(I+\mathrm{D}_{k}^{*}\right) \prod_{i=k+1}^{m} \mathscr{R}^{*}\left(t_{i}-t_{i-1}\right), \\
		\times\left(\mathrm{I}+\mathrm{D}_{i}^{*}\right) \mathscr{R}^{*}\left(b-t_{m}\right) \varphi, & t_{k-1}<t \leq t_{k}\end{cases} \\
	\mathrm{E}_{k}^{*} \psi\left(t_{k}^{+}\right) & = \begin{cases}\mathrm{E}_{m}^{*} \mathscr{R}^{*}\left(b-t_{m}\right) \varphi, & k=m, \\
		\mathrm{E}_{k}^{*} \prod_{i=k+1}^{m} \mathscr{R}^{*}\left(t_{i}-t_{i-1}\right) \\
		\times\left(\mathrm{I}+\mathrm{D}_{i}^{*}\right) \mathscr{R}^{*}\left(b-t_{m}\right) \varphi, & k=m-1, \ldots, 1\end{cases}
\end{aligned}
$$

The controllability Gramian operator is defined as:
	\begin{align*}
		\mathrm{M}\mathrm{M}^{*}=\Theta_{0}^{t_{m}}+\Gamma_{t_{m}}^{b}+\widetilde{\Theta}_{0}^{t_{m}}+\widetilde{\Gamma}_{t_{m}}^{b},
	\end{align*}
	where operator $M^{*}$ is the adjoint of $M$ 
	and $\Gamma_{t_{m}}^{b}, \widetilde{\Gamma}_{t_{m}}^{b}, \Theta_{0}^{t_{m}}, \widetilde{\Theta}_{0}^{t_{m}}: \mathbb{X} \rightarrow \mathbb{X}$ are non-negative operators and defined as follows:
	\begin{align*}
		\Gamma_{t_{m}}^{b}:= & \int_{t_{m}}^{b} \mathscr{R}(b-s) \mathrm{B} \mathrm{B}^{*} \mathscr{R}(b-s) d s, \quad \widetilde{\Gamma}_{t_{m}}^{b}:=\mathscr{R}\left(b-t_{m}\right) \mathrm{E}_{m} \mathrm{E}_{m}^{*} \mathscr{R}^{*}\left(b-t_{m}\right), 
	\end{align*}
	
	\begin{align*}
		\Theta_{0}^{t_{m}}:= & \mathscr{R}\left(b-t_{m}\right)\\ &\times \sum_{i=1}^{m} \prod_{j=m}^{i+1}\left(\mathrm{I}+\mathrm{D}_{j}\right) \mathscr{R}\left(t_{j}-t_{j-1}\right)\left(\mathrm{I}+\mathrm{D}_{i}\right) \int_{t_{i-1}}^{t_{i}} \mathscr{R}\left(t_{i}-s\right) \mathrm{B}\mathrm{B}^{*} \mathscr{R}^{*}\left(t_{k}-s\right) d s \\
		& \times\left(\mathrm{I}+\mathrm{D}_{i}^{*}\right) \prod_{k=i+1}^{m} \mathscr{R}^{*}\left(t_{k}-t_{k-1}\right)\left(\mathrm{I}+\mathrm{D}_{k}^{*}\right) \mathscr{R}^{*}\left(b-t_{m}\right), \\
		\widetilde{\Theta}_{0}^{t_{m}}:= & T\left(b-t_{m}\right) \sum_{i=2}^{m} \prod_{j=m}^{i}\left(I+\mathrm{D}_{j}\right) T\left(t_{j}-t_{j-1}\right) \mathrm{E}_{i-1} \mathrm{E}_{i-1}^{*} \\
		& \times \prod_{k=i}^{m} T^{*}\left(t_{k}-t_{k-1}\right)\left(I+\mathrm{D}_{k}^{*}\right) T^{*}\left(b-t_{m}\right) .
	\end{align*}
	The operators $\mathscr{R}^{*}$, $\mathrm{B}^{*}$, $\mathrm{D}_{k}^{*}$ and $\mathrm{E}_{m}^{*}$ are the adjoint operators of $\mathscr{R}$, $\mathrm{B}$, $\mathrm{D}_{k}$ and $\mathrm{E}_{k}$ respectively.
	
 Throughout the remainder of this paper, unless stated otherwise, we assume that $\mathbb{X}$
is a separable and reflexive Banach space whose dual $\mathbb{X}^{*}$ is uniformly convex. Now we proceed to prove some results on impulsive resolvent-like operator $ \left(\alpha\mathrm{I}+\left(\Theta_{0}^{t_{m}}+\Gamma_{t_{m}}^{b}+\widetilde{\Theta}_{0}^{t_{m}}+\widetilde{\Gamma}_{t_{m}}^{b}\right)\mathscr{J}\right)^{-1}$.

	\begin{lem}
		For $y \in \mathbb{X}$ and $\alpha>0$, the equation
		\begin{equation}\label{eq2.5}
			\alpha x_{\alpha}+ \left(\Theta_{0}^{t_{m}}+\Gamma_{t_{m}}^{b}+\widetilde{\Theta}_{0}^{t_{m}}+\widetilde{\Gamma}_{t_{m}}^{b}\right)\mathscr{J}\left[x_{\alpha}\right]=\alpha y ,
		\end{equation}
		has a unique solution
		\begin{equation*}
				x_{\alpha}=x_{\alpha}(y)=\alpha\left(\alpha \mathrm{I}+\left(\Theta_{0}^{t_{m}}+\Gamma_{t_{m}}^{b}+\widetilde{\Theta}_{0}^{t_{m}}+\widetilde{\Gamma}_{t_{m}}^{b}\right) \mathscr{J}\right)^{-1}(y) .
		\end{equation*}
		Moreover,
		\begin{equation} \label{eq2.6}
			\left\|x_{\alpha}(y)\right\|_{\mathbb{X}}=\left\|\mathscr{J}\left[x_{\alpha}(y)\right]\right\|_{\mathbb{X}^{*}} \leq\|y\|_{\mathbb{X}} .
		\end{equation}
	\end{lem}
	\begin{proof}
		As we can check that the operators $\Theta_{0}^{t_{m}}$, $\Gamma_{t_{m}}^{b}$, $\widetilde{\Theta}_{0}^{t_{m}}$, $\widetilde{\Gamma}_{t_{m}}^{b}$ are symmetric and non-negative operators, therefore $\mathrm{M}\mathrm{M}^{*}$ is a symmetric and non-negative operator. Also $\mathrm{M}$ is linear and bounded and this imply that $\mathrm{M}^{*}$ is also linear and bounded and hence $\mathrm{M}\mathrm{M}^{*}$ is linear and bounded opearator. Since every linear and bounded operator defined on a normed space is continuous therefore, $\mathrm{M}\mathrm{M}^{*}$ is a continuous operator. Now we can prove the existence of unique solution of \eqref{eq2.5} and \eqref{eq2.6} in similar way as in lemma 2.2 \cite{mahmudov2003approximate}. 
	\end{proof}
	
	\begin{lem}\label{l10}
		 The operator\newline $\left(\alpha \mathrm{I}+\left(\Theta_{0}^{t_{m}}+\Gamma_{t_{m}}^{b}+\widetilde{\Theta}_{0}^{t_{m}}+\widetilde{\Gamma}_{t_{m}}^{b}\right) \mathscr{J}\right)^{-1}: \mathbb{X} \rightarrow \mathbb{X}, \alpha>0$ is uniformly continuous in every bounded subset of $\mathbb{X}$.	
	\end{lem}
	\begin{proof}
		 For $\alpha >0$, we can prove the uniform continuity of operator \newline$\left(\alpha \mathrm{I}+\left(\Theta_{0}^{t_{m}}+\Gamma_{t_{m}}^{b}+\widetilde{\Theta}_{0}^{t_{m}}+\widetilde{\Gamma}_{t_{m}}^{b}\right) \mathscr{J}\right)^{-1}$ as proved in \cite{arora2025controllability}.
	\end{proof}
\begin{thm}
	The following statements are equivalent for the linear system \eqref{LP}:
	\begin{enumerate}
		\item[(i)] The operator sum $\Theta_{0}^{t_{m}} + \Gamma_{t_{m}}^{b} + \widetilde{\Theta}_{0}^{t_{m}} + \widetilde{\Gamma}_{t_{m}}^{b}$ is strictly positive.
		
		\item[(ii)] For any $h \in \mathbb{X}$, the sequence $J(x_\alpha(h))$, with $x_\alpha(h) = \alpha(\alpha I + \mathrm{M}\mathrm{M}^{*} J)^{-1}(h)$, converges weakly to zero in $\mathbb{X}^*$ as $\alpha \to 0^{+}$.
		
		\item[(iii)] For any $y \in \mathbb{X}$, the family $x_\alpha(y) = \alpha(\alpha I + (\Theta_{0}^{t_{m}} + \Gamma_{t_{m}}^{b} + \widetilde{\Theta}_{0}^{t_{m}} + \widetilde{\Gamma}_{t_{m}}^{b}) \mathscr{J})^{-1}(y)$ converges strongly to zero in $\mathbb{X}$ as $\alpha \to 0^{+}$.
	\end{enumerate}
\end{thm}

\begin{proof}
	It is known that $\mathrm{M}\mathrm{M}^{*} = \Theta_{0}^{t_{m}} + \Gamma_{t_{m}}^{b} + \widetilde{\Theta}_{0}^{t_{m}} + \widetilde{\Gamma}_{t_{m}}^{b}$ is a symmetric operator.
	
	\noindent\textbf{(i) $\Rightarrow$ (ii):} Assume that the operator $\Gamma := \mathrm{M}\mathrm{M}^{*}$ is strictly positive. From equation \eqref{eq2.6}, we can extract a weakly convergent subsequence $x_\alpha$ such that there exists $\bar{x}^* \in \mathbb{X}^*$ satisfying
	\[
	\langle \mathscr{J}(x_\alpha), x \rangle \to \langle \bar{x}^*, x \rangle \quad \text{as } \alpha \to 0^{+}, \quad \forall x \in \mathbb{X}.
	\]
	Since $\mathscr{J}$ is bijective, there exists a unique $\bar{x} \in \mathbb{X}$ such that $\bar{x}^* = \mathscr{J}(\bar{x})$. Using equation \eqref{eq2.5}, we have
	\[
	\alpha \langle \mathscr{J}(\bar{x}), x_\alpha \rangle + \langle \mathscr{J}(\bar{x}), \Gamma \mathscr{J}(x_\alpha) \rangle = \alpha \langle \mathscr{J}(\bar{x}), h \rangle.
	\]
	Taking the limit as $\alpha \to 0^{+}$, we obtain
	\[
	\langle \mathscr{J}(\bar{x}), \Gamma \mathscr{J}(\bar{x}) \rangle = 0.
	\]
	By the positivity of $\Gamma$, it follows that $\mathscr{J}(\bar{x}) = 0$ and hence $\bar{x} = 0$. Thus, all weakly convergent subsequences converge to zero, implying that $\mathscr{J}(x_\alpha)$ converges weakly to zero.
	
	\noindent\textbf{(ii) $\Rightarrow$ (iii):} Assume that $\mathscr{J}(x_\alpha) \rightharpoonup 0$ weakly as $\alpha \to 0^{+}$. From equation \eqref{eq2.5}, we write
	\begin{align*}
		\alpha \langle \mathscr{J}(x_\alpha), x_\alpha \rangle + \langle \mathscr{J}(x_\alpha), \Gamma \mathscr{J}(x_\alpha) \rangle &= \alpha \langle \mathscr{J}(x_\alpha), h \rangle, \\
		\alpha \|x_\alpha\|^2 + \langle \mathscr{J}(x_\alpha), \Gamma \mathscr{J}(x_\alpha) \rangle &= \alpha \langle \mathscr{J}(x_\alpha), h \rangle.
	\end{align*}
	Dividing by $\alpha$ and taking the limit yields
	\[
	\lim_{\alpha \to 0^{+}} \|x_\alpha\|^2 = \lim_{\alpha \to 0^{+}} \|\mathscr{J}(x_\alpha)\|^2 \leq \lim_{\alpha \to 0^{+}} \langle \mathscr{J}(x_\alpha), h \rangle = 0.
	\]
	Therefore, $x_\alpha \to 0$ strongly in $\mathbb{X}$.
	
	\noindent\textbf{(iii) $\Rightarrow$ (i):} Suppose for contradiction that $\Gamma$ is not strictly positive. Then there exists $x^* \neq 0$ such that
	\[
	\langle x^*, \Gamma x^* \rangle = 0.
	\]
	By Lemma 2.1 of \cite{mahmudov2003approximate}, this implies
	\[
	\langle x^*, \mathrm{M}^*\mathrm{M}x^* \rangle = \|\mathrm{M} x^*\|^2 = 0 \quad \Rightarrow \quad \mathrm{M}x^* = 0 \Rightarrow \Gamma x^* = 0.
	\]
	Since $\mathscr{J}$ is bijective, there exists $x \in \mathbb{X}$ such that $\mathscr{J}(x) = x^*$. Taking $h = x$ in \eqref{eq2.5}, we get
	\[
	\alpha x + \Gamma \mathscr{J}(x) = \alpha h.
	\]
	Because $\Gamma \mathscr{J}(x) = 0$, we conclude that $x_\alpha(h) = h$ for all $\alpha > 0$. Hence,
	\[
	\lim_{\alpha \to 0^{+}} x_\alpha(h) = h \neq 0,
	\]
	contradicting the assumption that $x_\alpha(h) \to 0$. Therefore, $\Gamma = \mathrm{M}\mathrm{M}^{*}$ must be strictly positive.
\end{proof}

	\begin{rem}
		From above, it is clear that condition (iii) of the above theorem holds if and only if $\|M^{*}\|>0$ for all non-zero $x\in \mathbb{X}$. This means that (iii) holds if and only if the linear system \eqref{LP} is approximately controllable.
	\end{rem}
	
	\begin{thm}
		Let $\mathrm{M}\mathrm{M}^{*}: \mathbb{X}^{*} \rightarrow \mathbb{X}$ be a symmetric and positive operator, and let $h: \mathbb{X} \rightarrow \mathbb{X}$ be a nonlinear mapping. Suppose that for each $\alpha > 0$, the equation
		\begin{equation}\label{eq2.7}
			\alpha x_{\alpha} + \mathrm{M}\mathrm{M}^{*} \mathscr{J}(x_{\alpha}) = \alpha h(x_{\alpha})
		\end{equation}
		has a solution $x_\alpha$, and that the sequence $\{h(x_\alpha)\}$ converges strongly to some $\bar{h} \in \mathbb{X}$ as $\alpha \to 0^+$. Then there exists a subsequence of $\{x_\alpha\}$ that converges strongly to zero in $\mathbb{X}$ as $\alpha \to 0^+$.
	\end{thm}
	
	\begin{proof}
		Due to the assumption that $h(x_\alpha) \to \bar{h}$ strongly as $\alpha \to 0^+$, it follows that the sequence $\{x_\alpha\}$ is bounded in $\mathbb{X}$. Indeed, taking norms in \eqref{eq2.7} and recalling that $\mathscr{J}$ is an isometry yields
		\[
		\|x_\alpha\| = \|\mathscr{J}(x_\alpha)\| \leq \|h(x_\alpha)\| \leq C,
		\]
		for some constant $C > 0$ independent of $\alpha$. Hence, by the Banach-Alaoglu theorem, there exists a weakly convergent subsequence (still denoted by $\{x_\alpha\}$) such that
		\[
		\mathscr{J}(x_\alpha) \rightharpoonup \mathscr{J}(\bar{x}_0) \quad \text{in } \mathbb{X}^* \text{ as } \alpha \to 0^+,
		\]
		for some $\bar{x}_0 \in \mathbb{X}$. 
		
		Applying the duality pairing with $\mathscr{J}(\bar{x}_0)$ to both sides of \eqref{eq2.7}, we obtain:
		\[
		\alpha \langle \mathscr{J}(\bar{x}_0), x_\alpha \rangle + \langle \mathscr{J}(\bar{x}_0), \mathrm{M}\mathrm{M}^{*} \mathscr{J}(x_\alpha) \rangle = \alpha \langle \mathscr{J}(\bar{x}_0), h(x_\alpha) \rangle.
		\]
		Taking the limit as $\alpha \to 0^+$ and using the weak convergence and the continuity of the involved operators, we get:
		\[
		\langle \mathscr{J}(\bar{x}_0), \mathrm{M}\mathrm{M}^{*} \mathscr{J}(\bar{x}_0) \rangle = 0.
		\]
		Since $\mathrm{M}\mathrm{M}^{*}$ is positive and symmetric, this implies $\mathscr{J}(\bar{x}_0) = 0$, i.e., $\bar{x}_0 = 0$. Therefore, we conclude that $\mathscr{J}(x_\alpha) \rightharpoonup 0$ in $\mathbb{X}^*$.
		
		To prove strong convergence, we now take the duality pairing of both sides of \eqref{eq2.7} with $\mathscr{J}(x_\alpha)$:
		\[
		\alpha \|x_\alpha\|^2 + \langle \mathscr{J}(x_\alpha), \mathrm{M}\mathrm{M}^{*} \mathscr{J}(x_\alpha) \rangle = \alpha \langle \mathscr{J}(x_\alpha), h(x_\alpha) \rangle.
		\]
		Dividing by $\alpha$ and passing to the limit as $\alpha \to 0^+$, we obtain:
		\[
		\lim_{\alpha \to 0^+} \|x_\alpha\|^2 \leq \lim_{\alpha \to 0^+} \langle \mathscr{J}(x_\alpha), h(x_\alpha) \rangle.
		\]
		We decompose the right-hand side:
		\[
		\langle \mathscr{J}(x_\alpha), h(x_\alpha) \rangle = \langle \mathscr{J}(x_\alpha), h(x_\alpha) - \bar{h} \rangle + \langle \mathscr{J}(x_\alpha), \bar{h} \rangle.
		\]
		The first term vanishes due to the strong convergence of $h(x_\alpha) \to \bar{h}$ and the boundedness of $\mathscr{J}(x_\alpha)$, and the second term tends to zero because $\mathscr{J}(x_\alpha) \rightharpoonup 0$. Thus,
		\[
		\lim_{\alpha \to 0^+} \|x_\alpha\|^2 = 0,
		\]
		which implies strong convergence of a subsequence: $x_\alpha \to 0$ in $\mathbb{X}$ as $\alpha \to 0^+$.
	\end{proof}

\section{approximate controllability of semilinear systems}
In this section, we investigate the existence and approximate controllability of the semilinear system described by \eqref{P}. 

We start by considering the control function $u_{\alpha}$ for $\alpha>0$  as follows:
	\begin{align}\label{eq4.8}
		\begin{split}
			u_{\alpha}(s)=\bigg(\sum_{k=1}^{m}\mathrm{B}^{*}\mathscr{R}^{*}(t_{k}-s)\prod_{i=k+1}^{m} \mathscr{R}^{*}(t_{i}-t_{i-1})\mathscr{R}^{*}(b-t_{m})\chi(t_{k-1},t_{k})\\
			+\mathrm{B}^{*}\mathscr{R}^{*}(b-s)\chi(t_{m},b)\bigg)\widehat{\varphi}_{\alpha},
				\end{split}
		\end{align}
	where 
	
	$$v_{m}=\mathrm{E}_{m}^{*}\mathscr{R}^{*}(b-t_{m})\widehat{\varphi}_{\alpha},\quad\mbox{and}\;\;
			v_{k}=\mathrm{E}_{k}^{*}\prod_{i=k}^{m} \mathscr{R}^{*}(t_{i}-t_{i-1})(\mathrm{I}+\mathrm{D}_{i}^{*})\mathscr{R}^{*}(b-t_{m})\widehat{\varphi}_{\alpha},$$
	with 
	\begin{align*}
		\widehat{\varphi}_{\alpha}=&\mathscr{J}\left[\left(\alpha \mathrm{I}+\left(\Theta_{0}^{t_{m}}+\Gamma_{t_{m}}^{b}+\widetilde{\Theta}_{0}^{t_{m}}+\widetilde{\Gamma}_{t_{m}}^{b}\right) \mathscr{J}\right)^{-1} \sigma\left(x(\cdot)\right)\right],
	\end{align*} 
where
	\begin{align*}
		\sigma\left(x(\cdot)\right)&=h-\mathscr{R}\left(b-t_{m}\right) \prod_{j=m}^{1}(\mathrm{I}+\mathrm{D}_{j}) \mathscr{R}\left(t_{j}-t_{j-1}\right) \psi(0)\\ & \qquad-\int_{t_{m}}^{b} \mathscr{R}(b-s)\left[f(s,\tilde{x}_{s})+f_{1}^{\prime}(s)+f_{2}(s)\right] d s\\
		&\qquad-\mathscr{R}\left(b-t_{m}\right)\sum_{i=1}^{m} \prod_{j=m}^{i+1}\left(\mathrm{I}+\mathrm{D}_{j}\right) \mathscr{R}\left(t_{j}-t_{j-1}\right)\left(\mathrm{I}+\mathrm{D}_{i}\right)\int_{t_{i-1}}^{t_{i}} \mathscr{R}\left(t_{i}-s\right)\\ & \qquad\qquad \qquad \qquad \qquad \qquad \times\left[f(s,\tilde{x}_{s})+f_{1}^{\prime}(s)+f_{2}(s)\right] d s.
	\end{align*}

	In order to establish the subsequent results presented in this work, it is essential to impose the following assumption, which will play a critical role in the development and validation of the forthcoming theorems:
\begin{enumerate}
	\item[(H1)] For each $y \in \mathbb{X}$, consider the family of elements defined by 
	\[
	x_{\alpha}(y) = \alpha \left(\alpha \mathrm{I} + \left(\Theta_{0}^{t_m} + \Gamma_{t_m}^{b} + \widetilde{\Theta}_{0}^{t_m} + \widetilde{\Gamma}_{t_m}^{b}\right) \mathscr{J}\right)^{-1}(y),
	\]
	where $x_{\alpha}(y)$ solves equation \eqref{eq2.7}. We assume that $x_{\alpha}(y)$ converges strongly to zero in $\mathbb{X}$ as $\alpha \to 0^+$.
	
	\item[(H2)] There exists $\alpha_0 \in \rho(\mathrm{A})$ such that the resolvent operator $\mathrm{R}(\alpha_0, \mathrm{A})$ is compact.
	
	\item[(H3)] The nonlinear function $f: J \times \mathfrak{B} \rightarrow \mathbb{X}$ satisfies the following:
	\begin{enumerate}
		\item For almost every $t \in J$, the map $\phi \mapsto f(t, \cdot)$ is continuous on $\mathfrak{B}$.
		\item For every $\phi \in \mathfrak{B}$, the function $t \mapsto f(t, \phi)$ is strongly measurable.
		\item There exists a function $\gamma \in \mathrm{L}^1(J; \mathbb{R}^+)$ such that
		\[
		\|f(t, \phi)\|_{\mathbb{X}} \leq \gamma(t) \quad \text{for almost every } t \in J \text{ and all } \phi \in \mathfrak{B}.
		\]
	\end{enumerate}
\end{enumerate}

	\begin{thm} \label{theorem4.1}
		Assume that hypotheses \textnormal{(H2)}–\textnormal{(H3)} are satisfied. Then for any $h \in \mathbb{X}$ and each fixed $\alpha > 0$, the semilinear system \eqref{P} governed by the control given in \eqref{eq4.8} admits at least one mild solution on the interval $J$.
	\end{thm}
	
	\begin{proof}
		We begin by defining the set 
		\[
		\mathcal{Q}_{\psi} = \left\{ x \in \mathcal{PC}(J; \mathbb{X}) : x(0) = \psi(0) \right\},
		\]
		equipped with the norm $\|\cdot\|_{\mathcal{PC}}$. Consider the closed ball
		\[
		\mathcal{B}_r = \left\{ x \in \mathcal{Q}_{\psi} : \|x\|_{\mathcal{PC}} \leq r \right\},
		\]
		for a positive constant $r$. This subset of $\mathcal{PC}(J; \mathbb{X})$ will serve as the domain for the operator whose fixed point corresponds to a mild solution of system \eqref{P}.
	
	For $\alpha>0$, we introduce an operator $G_{\alpha}: \mathcal{Q}_{\psi} \rightarrow \mathcal{Q}_{\psi}$ defined as
	$$
	\left(G_{\alpha} x\right)(t)=
	\left\{\begin{array}{l}
		\mathscr{R}(t) \psi(0)+\int_{0}^{t} \mathscr{R}(t-s)\left[\mathrm{B} u_{\alpha}(s)+f(s,\tilde{x}_{s})\right] ds\\ \qquad \qquad +\int_{0}^{t}\mathscr{R}(t-s)\left[f_{1}^{\prime}(s)+f_{2}(s)\right] ds, 0 \leq t \leq t_{1}, \\
		\mathscr{R}\left(t-t_{k}\right) x\left(t_{k}^{+}\right)+\int_{t_{k}}^{t} \mathscr{R}(t-s)\left[\mathrm{B} u_{\alpha}(s)+f(s,\tilde{x}_{s})\right] ds\\ \qquad\qquad+\int_{t_k}^{t}\mathscr{R}(t-s)\left[f_{1}^{\prime}(s)+f_{2}(s)\right] ds, t_{k}<t \leq t_{k+1}, k=1,2, \ldots, m,
	\end{array}\right.
	$$
	where
	$$
	\begin{aligned}
		x\left(t_{k}^{+}\right) & =\prod_{j=k}^{1}\left(\mathcal{I}+\mathrm{D}_{j}\right) \mathscr{R}\left(t_{j}-t_{j-1}\right) \psi(0) \\
		&\qquad +\sum_{i=1}^{k} \prod_{j=k}^{i+1}\left(\mathcal{I}+\mathrm{D}_{j}\right) \mathscr{R}\left(t_{j}-t_{j-1}\right)\left(\mathcal{I}+\mathrm{D}_{i}\right) \int_{t_{i-1}}^{t_{i}} \mathscr{R}\left(t_{i}-s\right) \mathrm{B} u_{\alpha}(s) d s \\
		&\qquad +\sum_{i=1}^{k} \prod_{j=k}^{i+1}\left(\mathcal{I}+\mathrm{D}_{j}\right) \mathscr{R}\left(t_{j}-t_{j-1}\right)\left(\mathcal{I}+\mathrm{D}_{i}\right) \int_{t_{i-1}}^{t_{i}} \mathscr{R}\left(t_{i}-s\right)\left[f(s,\tilde{x}_{s})+f_{1}^{\prime}(s)+f_{2}(s)\right]  d s \\
		& \qquad +\sum_{i=2}^{k} \prod_{j=k}^{i}\left(\mathcal{I}+\mathrm{D}_{j}\right) \mathscr{R}\left(t_{j}-t_{j-1}\right) \mathrm{E}_{i-1} v_{i-1}+\mathrm{E}_{k} v_{k}.
	\end{aligned}
	$$
	
Here, $u_{\alpha}(\cdot)$ is defined as in \eqref{eq4.8}. From the definition of the operator $G_{\alpha}$, it becomes clear that proving the existence of a mild solution to the system \eqref{P} reduces to showing that $G_{\alpha}$ possesses a fixed point. This will be established through the following steps.

\textbf{Step 1:} We begin by showing that $G_{\alpha}\left(\mathcal{B}_{r}\right) \subset \mathcal{B}_{r}$ for some $r = r(\alpha) > 0$. Assume, on the contrary, that this inclusion does not hold. Then, for every $\alpha > 0$ and each $r > 0$, there exists an element $x_r(\cdot) \in \mathcal{B}_r$ such that $\left\|\left(G_{\alpha} x_r\right)(t)\right\|_{\mathbb{X}} > r$ for some $t \in [0, b]$ (where $t$ may vary with $r$). We now proceed to compute
	$$
	\begin{aligned}
		\|\sigma\left(x(\cdot)\right)\|_{\mathbb{X}} & \leq\|h\|_{\mathbb{X}}+\left\|\mathscr{R}\left(b-t_{m}\right)\right\|_{\mathbb{X}} \sum_{j=m}^{1}\left(1+\left\|\mathrm{D}_{j}\right\|_{\mathbb{X}}\right)\left\|\mathscr{R}\left(t_{j}-t_{j-1}\right)\right\|_{\mathbb{X}}\left\|\psi(0)\right\|_{\mathbb{X}} \\
		& \qquad +\left\|\mathscr{R}\left(b-t_{m}\right)\right\|_{\mathbb{X}} \sum_{i=1}^{m} \prod_{j=m}^{i+1}\left(1+\left\|\mathrm{D}_{j}\right\|_{\mathbb{X}}\right)\left\|\mathscr{R}\left(t_{j}-t_{j-1}\right)\right\|_{\mathbb{X}} \\
		& \qquad \qquad \times\left(1+\left\|\mathrm{D}_{i}\right\|_{\mathbb{X}}\right) \int_{t_{i-1}}^{t_{i}}\left\|\mathscr{R}\left(t_{i}-s\right) \left[f(s,\tilde{x}_{s})+f_{1}^{\prime}(s)+f_{2}(s)\right]\right\|_{\mathbb{X}} d s\\   & \qquad +\int_{t_{m}}^{b}\|\mathscr{R}(b-s) \left[f(s,\tilde{x}_{s})+f_{1}^{\prime}(s)+f_{2}(s)\right]\|_{\mathbb{X}} d s, \\
		& \leq\|h\|_{\mathbb{X}}+m M_{\mathscr{R}}^{2}\left(1+M_{\mathrm{D}}\right)\left\|\psi(0)\right\|_{\mathbb{X}}\\ & \qquad+\sum_{k=1}^{m}\left(1+M_{\mathrm{D}}\right)^{k} M_{\mathscr{R}}^{k} \int_{t_{k-1}}^{t_{k}}\left\|\mathscr{R}\left(t_{k}-s\right) \left[f(s,\tilde{x}_{s})+f_{1}^{\prime}(s)+f_{2}(s)\right]\right\|_{\mathbb{X}} d s \\
		& \qquad+\int_{t_{m}}^{b}\|\mathscr{R}(b-s) \left[f(s,\tilde{x}_{s})+f_{1}^{\prime}(s)+f_{2}(s)\right]\|_{\mathbb{X}} d s, \\
		& \leq\|h\|_{\mathbb{X}}+m M_{\mathscr{R}}^{2}\left(1+M_{\mathrm{D}}\right)\left\|\psi(0)\right\|_{\mathbb{X}}\\ & \qquad +\sum_{k=0}^{m}\left(1+M_{\mathrm{D}}\right)^{k} M_{\mathscr{R}}^{k+1} \left(\|\gamma\|_{L^{1}\left(J ; \mathbb{R}^{+}\right)}+\left\|f_{1}^{\prime}+f_{2}\right\|_{L^{1}(J ; \mathbb{X})}\right) , \\
		& =\beta,
	\end{aligned}
	$$
	
	Then, we have
	\begin{align}\label{eq4.9}
		\left \|u_{\alpha}(s)\right\|_{\mathbb{U}} &=\bigg\|\bigg(\sum_{k=1}^{m} \mathrm{B}^{*} \mathscr{R}^{*}\left(t_{k}-s\right) \prod_{i=k+1}^{m} \mathscr{R}^{*}\left(t_{i}-t_{i-1}\right)\nonumber \\	& \qquad\times \mathscr{R}^{*}\left(b-t_{m}\right) \chi_{\left(t_{k-1}, t_{k}\right)}+\mathrm{B}^{*} \mathscr{R}^{*}(b-s) \chi_{\left(t_{m}, b\right)}\bigg) \widehat{\varphi}_{\alpha} \bigg\|_{\mathbb{U}},\nonumber \\
		& \leq \tilde{M}\left\|\mathscr{J}\left[\left(\alpha \mathrm{I}+\left(\Theta_{0}^{t_{m}}+\Gamma_{t_{m}}^{b}+\widetilde{\Theta}_{0}^{t_{m}}+\widetilde{\Gamma}_{t_{m}}^{b}\right) \mathscr{J}\right)^{-1}\sigma\left(x(\cdot)\right)\right]\right\|_{\mathbb{X}},\nonumber\\
		& \leq \frac{\tilde{M}}{\alpha} \|\sigma\left(x(\cdot)\right)\|_{\mathbb{X}},\nonumber\\
		& \leq \frac{\tilde{M} \beta}{\alpha}, 
	\end{align}
	where $\tilde{M}=M_{\mathrm{B}} \sum_{k=1}^{m+1} M_{\mathscr{R}}^{k}.$ 	Considering $t \in\left[0, t_{1}\right)$, and applying the estimate \eqref{eq4.9}, 
	
	\begin{align}\label{eq4.10}
		\left\|\left(G_{\alpha} x_{r}\right)(t)\right\|_{\mathbb{X}} & \leq M_{\mathscr{R}}\left\|\psi(0)\right\|_{\mathbb{X}}+M_{\mathscr{R}} M_{\mathrm{B}} \int_{0}^{t}\left\|u_{\alpha}(s)\right\|_{\mathbb{U}} d s+M_{\mathscr{R}}\left(\|\gamma\|_{L^{1}\left(J ; \mathbb{R}^{+}\right)}+\left\|f_{1}^{\prime}+f_{2}\right\|_{L^{1}(J ; \mathbb{X})}\right), \nonumber \\
		& \leq M_{\mathbb{X}}\left\|\psi(0)\right\|_{\mathbb{X}}+\frac{\tilde{M} \beta t_{1}}{\alpha}+M_{\mathscr{R}} \left(\|\gamma\|_{L^{1}\left(J ; \mathbb{R}^{+}\right)}+\left\|f_{1}^{\prime}+f_{2}\right\|_{L^{1}(J ; \mathbb{X})}\right),\nonumber \\
		& \leq M_{\mathscr{R}}\left\|\psi(0)\right\|_{\mathbb{X}}+\frac{\tilde{M} \beta b}{\alpha}+M_{\mathscr{R}} \left(\|\gamma\|_{L^{1}\left(J ; \mathbb{R}^{+}\right)}+\left\|f_{1}^{\prime}+f_{2}\right\|_{L^{1}(J ; \mathbb{X})}\right). 
	\end{align}

	For $t \in\left[t_{k}, t_{k+1}\right].$ for $k=1, \ldots, m$, we have
	
	$$
	\begin{aligned}
		\left\|\left(G_{\alpha} x_{r}\right)(t)\right\|_{\mathbb{X}} & \leq M_{\mathscr{R}}\left\|x_{r}\left(t_{k}^{+}\right)\right\|_{\mathbb{X}}+\frac{\tilde{M} \gamma}{\alpha}+M_{\mathscr{R}}\left(\|\gamma\|_{L^{1}\left(J ; \mathbb{R}^{+}\right)}+\left\|f_{1}^{\prime}+f_{2}\right\|_{L^{1}(J ; \mathbb{X})}\right),  \\
		& \leq M_{\mathscr{R}}\left\|x_{r}\left(t_{k}^{+}\right)\right\|_{\mathbb{X}}+\frac{\tilde{M} \gamma b}{\alpha}+M_{\mathscr{R}} \left(\|\gamma\|_{L^{1}\left(J ; \mathbb{R}^{+}\right)}+\left\|f_{1}^{\prime}+f_{2}\right\|_{L^{1}(J ; \mathbb{X})}\right), 
	\end{aligned}
	$$
	
	where
	
	$$
	\begin{aligned}
		\left\|x_{r}\left(t_{k}^{+}\right)\right\|_{\mathbb{X}} &\leq\left(1+M_{\mathrm{D}}\right)^{k} M_{\mathscr{R}}^{k}\left\|\psi(0)\right\|_{\mathbb{X}}+\sum_{p=1}^{k}\left(1+M_{\mathrm{D}}\right)^{p} M_{\mathscr{R}}^{p} M_{\mathrm{B}} b\left\|u_{\alpha}\right\|_{\mathbb{X}} \\
		 & \qquad+\sum_{p=1}^{k}\left(1+M_{\mathrm{D}}\right)^{p} M_{\mathscr{R}}^{p} \left(\|\gamma\|_{L^{1}\left(J ; \mathbb{R}^{+}\right)}+\left\|f_{1}^{\prime}+f_{2}\right\|_{L^{1}(J ; \mathbb{X})}\right) \\& \qquad +\sum_{p=1}^{k-1}\left(1+M_{\mathrm{D}}\right)^{p} M_{\mathscr{R}} M_{\mathrm{E}} M_{V}+M_{\mathrm{E}} M_{V}, \\
		&\leq \left(1+M_{\mathrm{D}}\right)^{k} M_{\mathscr{R}}^{k}\left\|\psi(0)\right\|_{\mathbb{X}}+\sum_{p=1}^{k}\left(1+M_{\mathrm{D}}\right)^{p} M_{\mathscr{R}}^{p} M_{\mathrm{B}} b \frac{\tilde{M} \gamma}{\alpha} \\
		& \qquad+ \sum_{p=1}^{k}\left(1+M_{\mathrm{D}}\right)^{p} M_{\mathscr{R}}^{p} \left(\|\gamma\|_{L^{1}\left(J ; \mathbb{R}^{+}\right)}+\left\|f_{1}^{\prime}+f_{2}\right\|_{L^{1}(J ; \mathbb{X})}\right)\\ & \qquad +\sum_{p=1}^{k-1}\left(1+M_{\mathrm{D}}\right)^{p} M_{\mathscr{R}} M_{\mathrm{E}} M_{V}+M_{\mathrm{E}} M_{V}.
	\end{aligned}
	$$
	As a result, we get
	\begin{align}\label{eq4.11}
		\left\|\left(G_{\alpha} x_{r}\right)(t)\right\|_{\mathbb{X}} &\leq\left(1+M_{\mathrm{D}}\right)^{k} M_{\mathscr{R}}^{k+1}\left\|\psi(0)\right\|_{\mathbb{X}}+\sum_{p=1}^{k}\left(1+M_{\mathrm{D}}\right)^{p} M_{\mathscr{R}}^{p+1} M_{\mathrm{B}} b \frac{\tilde{M} \gamma}{\alpha} \nonumber \\
		& \qquad +\sum_{p=1}^{k}\left(1+M_{\mathrm{D}}\right)^{p} M_{\mathscr{R}}^{p+1} b\left(\|\gamma\|_{L^{1}\left(J ; \mathbb{R}^{+}\right)}+\left\|f_{1}^{\prime}+f_{2}\right\|_{L^{1}(J ; \mathbb{X})}\right) \\ & \qquad +\sum_{p=1}^{k-1}\left(1+M_{\mathrm{D}}\right)^{p} M_{\mathscr{R}}^{2} M_{\mathrm{E}} M_{V}+M_{\mathscr{R}} M_{\mathrm{E}} M_{V} \nonumber \\
		&\qquad +\frac{\tilde{M} \gamma b}{\alpha}+M_{\mathscr{R}} \left(\|\gamma\|_{L^{1}\left(J ; \mathbb{R}^{+}\right)}+\left\|f_{1}^{\prime}+f_{2}\right\|_{L^{1}(J ; \mathbb{X})}\right). 
	\end{align}
	%
	From the inequalities \eqref{eq4.10} and \eqref{eq4.11}, it is clear that for every $\alpha >0$, there is a large $r=r(\alpha)>0$ which gives $G_{\alpha}\left(\mathcal{B}_{r}\right) \subset \mathcal{B}_{r}$.
	
\textbf{Step 2:} We now show that the operator $G_{\alpha}$ is compact. This will be done by applying the infinite-dimensional version of the generalized Arzelà-Ascoli theorem (see Theorem 2.1 in \cite{wei2006nonlinear}). According to this result, the compactness of $G_{\alpha}$ can be concluded if the following conditions are satisfied:

\begin{enumerate}
	\item The image $G_{\alpha}(\mathcal{B}_{r})$ is uniformly bounded, which has already been addressed in Step 1.
	\item For any $x \in \mathcal{B}_{r}$, the mapping $t \mapsto \left(G_{\alpha} x\right)(t)$ is equicontinuous on each open subinterval $\left(t_{j}, t_{j+1}\right)$, where $j=0,1,\ldots,m$.
	\item The sets 
	\[
	W(t) = \left\{\left(G_{\alpha} x\right)(t) : x \in \mathcal{B}_{r},\ t \in J \setminus \{t_1, \ldots, t_n\} \right\},
	\]
	\[
	W\left(t_j^+\right) = \left\{ \left(G_{\alpha} x\right)(t_j^+) : x \in \mathcal{B}_{r} \right\}, \quad 
	W\left(t_j^-\right) = \left\{ \left(G_{\alpha} x\right)(t_j^-) : x \in \mathcal{B}_{r} \right\}, 
	\]
	for each $j = 1, \ldots, n$, are relatively compact in $\mathbb{X}$.
\end{enumerate}

To begin, we verify the equicontinuity of the family $\left\{ \left(G_{\alpha} x\right)(t) : x \in \mathcal{B}_{r} \right\}$ on the interval $(t_j, t_{j+1})$ for each $j = 0,1, \ldots, n$. Take two points $a_1, a_2 \in (0, t_1)$ such that $a_1 < a_2$, and let $x \in \mathcal{B}_r$. We now evaluate the expression:
	\begin{align}\label{eq4.12}
		&\left\|\left(G_{\alpha} x\right)\left(a_{2}\right)-\left(G_{\alpha} x\right)\left(a_{1}\right)\right\|_{\mathbb{X}}\nonumber\\ & \qquad \leq\left\|\mathscr{R}\left(a_{2}\right)-\mathscr{R}\left(a_{1}\right)\right\|_{L(\mathbb{X})}\|\psi(0)\|_{\mathbb{X}} +\left\|\int_{a_{1}}^{a_{2}} \mathscr{R}\left(a_{2}-s\right) \mathrm{B} u_{\alpha}(s) d s\right\|_{\mathbb{X}} \nonumber\\
		&\qquad \qquad +\left\|\int_{a_{1}}^{a_{2}} \mathscr{R}\left(a_{2}-s\right) \left[f(s,\tilde{x}_{s})+f_{1}^{\prime}(s)+f_{2}(s)\right] d s\right\|_{\mathbb{X}}\nonumber \\
		&\qquad \qquad +\left\|\int_{0}^{a_{1}}\left(\mathscr{R}\left(a_{2}-s\right)-\mathscr{R}\left(a_{1}-s\right)\right) \mathrm{B} u_{\alpha}(s) d s\right\|_{\mathbb{X}} \nonumber\\
		& \qquad \qquad +\left\|\int_{0}^{a_{1}}\left(\mathscr{R}\left(a_{2}-s\right)-\mathscr{R}\left(a_{1}-s\right)\right) \left[f(s,\tilde{x}_{s})+f_{1}^{\prime}(s)+f_{2}(s)\right] d s\right\|_{\mathbb{X}}, \nonumber \\
		& \qquad \leq\left\|\mathscr{R}\left(a_{2}\right)-\mathscr{R}\left(a_{1}\right)\right\|_{L(\mathbb{X})}\|\psi(0)\|_{\mathbb{X}} +\frac{M_{\mathscr{R}} M_{\mathrm{B}} \tilde{M} \beta }{\alpha}\left(a_{2}-a_{1}\right)\nonumber\\
		& \qquad\qquad+ M_{\mathscr{R}}\int_{a_{1}}^{a_{2}} \left\|\left[f(s,\tilde{x}_{s})+f_{1}^{\prime}(s)+f_{2}(s)\right]\right\|_{\mathbb{X}} d s \nonumber\\ 
		& \qquad\qquad+\sup_{t \in[0,a_1]} \left\|\mathscr{R}\left(a_{2}-s\right)-\mathscr{R}\left(a_{1}-s\right)\right\|_{L(\mathbb{X})}\int_{0}^{a_{1}} \left\|\left[\mathrm{B} u_{\alpha}(s)\right]\right\|_{\mathbb{X}} d s\nonumber\\ 
		& \qquad\qquad+\sup_{t \in[0,a_1]} \left\|\mathscr{R}\left(a_{2}-s\right)-\mathscr{R}\left(a_{1}-s\right)\right\|_{L(\mathbb{X})}\int_{0}^{a_{1}} \left\|\left[f(s,\tilde{x}_{s})+f_{1}^{\prime}(s)+f_{2}(s)\right]\right\|_{\mathbb{X}} d s.
	\end{align}

Similarly, for any $a_1, a_2 \in (t_k, t_{k+1})$ with $k = 1, \ldots, p$ and $a_1 < a_2$, and for each $x \in \mathcal{B}_r$, we can carry out the following estimation:	
	$$
	\begin{aligned}
		&\left\|\left(G_{\alpha} x\right)\left(a_{2}\right)-\left(G_{\alpha} x\right)\left(a_{1}\right)\right\|_{\mathbb{X}}\\  &\qquad \qquad \leq\left\|\mathscr{R}\left(a_{2}-t_{k}\right)-\mathscr{R}\left(a_{1}-t_{k}\right)\right\|_{L(\mathbb{X})}\left\|x\left(t_{k}^{+}\right)\right\|_{\mathbb{X}}  +\left\|\int_{a_{1}}^{a_{2}} \mathscr{R}\left(a_{2}-s\right) \mathrm{B} u_{\alpha}(s) d s\right\|_{\mathbb{X}} \\
		&  \qquad \qquad\qquad +\left\|\int_{a_{1}}^{a_{2}} \mathscr{R}\left(a_{2}-s\right) \left[f(s,\tilde{x}_{s})+f_{1}^{\prime}(s)+f_{2}(s)\right] d s\right\|_{\mathbb{X}} \\
		& \qquad \qquad \qquad +\left\|\int_{t_{k}}^{a_{1}}\left(\mathscr{R}\left(a_{2}-s\right)-\mathscr{R}\left(a_{1}-s\right)\right) \mathrm{B} u_{\alpha}(s) d s\right\|_{\mathbb{X}} \\
		&\qquad \qquad \qquad +\left\|\int_{t_{k}}^{a_{1}}\left(\mathscr{R}\left(a_{2}-s\right)-\mathscr{R}\left(a_{1}-s\right)\right) \left[f(s,\tilde{x}_{s})+f_{1}^{\prime}(s)+f_{2}(s)\right] d s\right\|_{\mathbb{X}} .
	\end{aligned}
	$$
For the above inequality $\left\|x\left(t_{k}^{+}\right)\right\|_{\mathbb{X}}$ can be estimated as
	$$
	\begin{aligned}
		 \left\|x\left(t_{k}^{+}\right)\right\|_{\mathbb{X}} &\leq\left(1+M_{\mathrm{D}}\right)^{k} M_{\mathscr{R}}^{k}\left\|\psi(0)\right\|_{\mathbb{X}}+\sum_{m=1}^{k}\left(1+M_{\mathrm{D}}\right)^{m} M_{\mathscr{R}}^{m} M_{\mathrm{B}} b \frac{\tilde{M} \beta}{\alpha} \\
		 &\qquad+ \sum_{m=1}^{k}\left(1+M_{\mathrm{D}}\right)^{m} M_{\mathscr{R}}^{m} b\left(\|\gamma\|_{L^{1}\left(J ; \mathbb{R}^{+}\right)}+\left\|f_{1}^{\prime}+f_{2}\right\|_{L^{1}(J ; \mathbb{X})}\right)\\ & \qquad +\sum_{m=1}^{k-1}\left(1+M_{\mathrm{D}}\right)^{m} M_{\mathscr{R}} M_{\mathrm{E}} M_{V}+M_{\mathrm{E}} M_{V}= \Upsilon.
	\end{aligned}
	$$
	Then, we obtain
	
	\begin{align} \label{eq4.13}
		&\left\|\left(G_{\alpha} x\right)\left(a_{2}\right)-\left(G_{\alpha} x\right)\left(a_{1}\right)\right\|_{\mathbb{X}}\nonumber\\ & \qquad \leq\left\|\mathscr{R}\left(a_{2}-t_{k}\right)-\mathscr{R}\left(a_{1}-t_{k}\right)\right\|_{L(\mathbb{X})} \Upsilon \nonumber\\
		& \qquad\qquad +\frac{M_{\mathscr{R}} M_{\mathrm{B}} \tilde{M} \beta}{\alpha}\left(a_{2}-a_{1}\right)+M_{\mathscr{R}}\int_{a_{1}}^{a_{2}} \left\|\left[f(s,\tilde{x}_{s})+f_{1}^{\prime}(s)+f_{2}(s)\right]\right\|_{\mathbb{X}} d s \nonumber \\
		&\qquad\qquad +\sup_{t \in[0,a_1]} \left\|\mathscr{R}\left(a_{2}-s\right)-\mathscr{R}\left(a_{1}-s\right)\right\|_{L(\mathbb{X})}\int_{0}^{a_{1}} \left\|\left[\mathrm{B} u_{\alpha}(s)\right]\right\|_{\mathbb{X}} d s\nonumber\\ 
		&\qquad\qquad +\sup_{t \in[t_k,a_1]} \left\|\mathscr{R}\left(a_{2}-s\right)-\mathscr{R}\left(a_{1}-s\right)\right\|_{L(\mathbb{X})}\int_{t_k}^{a_{1}} \left\|\left[f(s,\tilde{x}_{s})+f_{1}^{\prime}(s)+f_{2}(s)\right]\right\|_{\mathbb{X}} d s.
	\end{align}
	
	Since $\mathscr{R}(t)$ is compact for every $t > 0$ (as established in Lemma 3.11 of \cite{dos2011existence}) and is also uniformly continuous on the interval $[0, b]$, it follows that the resolvent operator $\mathscr{R}(t)$ regularizes solutions as $t \to 0^{+}$. These properties imply that the right-hand sides of equations \eqref{eq4.12} and \eqref{eq4.13} tend to zero as $|a_2 - a_1| \to 0$. Hence, the family $\{ (G_{\alpha} x)(t) : x \in \mathcal{B}_{r} \}$ is equicontinuous for each $t \in (t_k, t_{k+1})$, where $k = 0, 1, \ldots, p$.
	
	Next, we verify that for each $t \in J$, the set
	\[
	W(t) = \left\{ (G_{\alpha} x)(t) : x \in \mathcal{B}_{r} \right\}
	\]
	is relatively compact in $\mathbb{X}$. This follows from the compactness of the operator $\mathscr{R}(t)$ for $t \geq 0$ (for further details, see Step 3 of Theorem 4.1 in \cite{arora2025controllability}). Consequently, the set $W(t)$ is relatively compact for all $t \in J$.
	
	By satisfying the conditions of the generalized Arzelà-Ascoli theorem, we conclude that the operator $G_{\alpha}$ is compact.
	
	\textbf{Step 3:} To establish the continuity of $G_{\alpha}$, let $\{x^n\}_{n=1}^{\infty} \subseteq \mathcal{B}_r$ be a sequence converging to $x$ in $\mathcal{B}_r$, that is,
	\begin{equation*}
		\lim _{n \rightarrow \infty}\left\|x^{n}-x\right\|_{\mathcal{PC}(J;\mathbb{X})}=0.
	\end{equation*}
	Using the axiom (A1), we estimate
	\begin{align*}
		\left\|\tilde{x}_{t}^{n}-\tilde{x}_{t}\right\|_{\mathcal{B}}&\le \Lambda(t)\sup_{0\le s\le t}\left\|x^{n}(s)-x(s)\right\|_{\mathbb{X}}\\
		&\le H_{1}\left\|x^{n}-x\right\|_{\mathcal{PC}(J;\mathbb{X})} \rightarrow 0 \text{ as } n\rightarrow \infty, \text{ for each} t\in J,
	\end{align*}
	where $\sup_{0\le s\le b} |\Lambda(t)| \le H_{1}$. 
	From the above convergence with Assumption $(H3)$ and Lebesgue's dominant convergence theorem, we deduce that 
	\begin{align*}
		\left\|\sigma(x^{n}(\cdot))-\sigma(x(\cdot))\right\|_{\mathbb{X}}&\le \left\|\mathscr{R}\left(b-t_{p}\right)\right\|_{\mathbb{X}} \sum_{i=1}^{m} \prod_{j=m}^{i+1}\left(1+\left\|\mathrm{D}_{j}\right\|_{\mathbb{X}}\right)\left\|\mathscr{R}\left(t_{j}-t_{j-1}\right)\right\|_{\mathbb{X}} \\
		&\qquad \qquad \times\left(1+\left\|\mathrm{\mathrm{D}}_{i}\right\|_{\mathbb{X}}\right) \int_{t_{i-1}}^{t_{i}}\left\|\mathscr{R}\left(t_{i}-s\right)\left(f\left(s, \tilde{x}^{n}_{s}\right)-f(s, \tilde{x}_{s})\right)\right\|_{\mathbb{X}} d s \\
		&\qquad +\int_{t_{m}}^{b}\left\|\mathscr{R}(b-s)\left(f\left(s,\tilde{x}^{n}_{s}\right)-f(s, \tilde{x}_{s})\right)\right\|_{\mathbb{X}} d s, \\
		& \leq \sum_{k=0}^{m}\left(1+M_{\mathrm{D}}\right)^{k} M_{\mathscr{R}}^{k+1} \int_{t_{k-1}}^{t_{k}}\left\|f\left(s,\tilde{x}^{n}_{s}\right)-f(s, \tilde{x}_{s})\right\|_{\mathbb{X}} d s \\
		& \rightarrow 0, n \rightarrow \infty .
	\end{align*}
	From the lemma \ref{l10},it follows that the maping $\left(\alpha \mathrm{I}+\left(\Theta_{0}^{t_{m}}+\Gamma_{t_{m}}^{b}+\widetilde{\Theta}_{0}^{t_{m}}+\widetilde{\Gamma}_{t_{m}}^{b}\right) \mathscr{J}\right)^{-1} : \mathbb{X} \rightarrow \mathbb{X}$ is uniformly continuous on every bounded subset of $\mathbb{X}$. Thus we have \begin{align*}
		&\left(\alpha \mathrm{I}+\left(\Theta_{0}^{t_{m}}+\Gamma_{t_{m}}^{b}+\widetilde{\Theta}_{0}^{t_{m}}+\widetilde{\Gamma}_{t_{m}}^{b}\right) \mathscr{J}\right)^{-1}\sigma (x^{n}(\cdot))\rightarrow \\ & \qquad\qquad \qquad\left(\alpha \mathrm{I}+\left(\Theta_{0}^{t_{m}}+\Gamma_{t_{m}}^{b}+\widetilde{\Theta}_{0}^{t_{m}}+\widetilde{\Gamma}_{t_{m}}^{b}\right) \mathscr{J}\right)^{-1}\sigma (x(\cdot))
	\end{align*}
	in $\mathbb{X}$ as $n\rightarrow \infty$.
	
	As the mapping $\mathscr{J}: \mathbb{X}\mapsto \mathbb{X}^{*}$ is demicontinuous, it follows that
	\begin{align*}
		&\mathscr{J}\left[\left(\alpha \mathrm{I}+\left(\Theta_{0}^{t_{m}}+\Gamma_{t_{m}}^{b}+\widetilde{\Theta}_{0}^{t_{m}}+\widetilde{\Gamma}_{t_{m}}^{b}\right) \mathscr{J}\right)^{-1}\sigma (x^{n}(\cdot))\right]\rightharpoonup\\ & \qquad\qquad \qquad \mathscr{J}\left[\left(\alpha \mathrm{I}+\left(\Theta_{0}^{t_{m}}+\Gamma_{t_{m}}^{b}+\widetilde{\Theta}_{0}^{t_{m}}+\widetilde{\Gamma}_{t_{m}}^{b}\right) \mathscr{J}\right)^{-1}\sigma (x(\cdot))\right] \text{ in } \mathbb{X}^{*} \text{ as } n\rightarrow \infty.
	\end{align*}
	By assumption (H2) and lemma 3.11 of \cite{dos2011existence}, it follows that the operator $\mathscr{R}(t)$ is compact for all $t>0$. Consequently, the operator $\mathscr{R}(t)^{*}$ is also comact for each $t>0$. Therefore, using the weak convergence alongside the compactness of the operator $\mathscr{R}(t)^{*}$, we can arrive
	\begin{align}\label{4.14}
		\bigg\|&\mathscr{R}(b-t_{m})^{*}\mathscr{J}\left[\left(\alpha \mathrm{I}+\left(\Theta_{0}^{t_{m}}+\Gamma_{t_{m}}^{b}+\widetilde{\Theta}_{0}^{t_{m}}+\widetilde{\Gamma}_{t_{m}}^{b}\right) \mathscr{J}\right)^{-1}\sigma (x^{n}(\cdot))\right]-\nonumber\\&\mathscr{R}(b-t_{m})^{*}\mathscr{J}\left[\left(\alpha \mathrm{I}+\left(\Theta_{0}^{t_{m}}+\Gamma_{t_{m}}^{b}+\widetilde{\Theta}_{0}^{t_{m}}+\widetilde{\Gamma}_{t_{m}}^{b}\right) \mathscr{J}\right)^{-1}\sigma (x(\cdot))\right]\bigg\|_{\mathbb{X}^{*}} \rightarrow 0 \text{ as } n\rightarrow \infty,
	\end{align}
	for $t\in [t_m,b]$. Using \eqref{eq4.8} and \eqref{4.14}, we receive
	\begin{align*}
		 \left\|u_{\alpha}^{n}(t)-u_{\alpha}(t)\right\|_{\mathbb{U}}&=\bigg\|\bigg(\sum_{k=1}^{m} \mathrm{B}^{*} \mathscr{R}^{*}\left(t_{k}-s\right) \prod_{i=k+1}^{m} \mathscr{R}^{*}\left(t_{i}-t_{i-1}\right)  \mathscr{R}^{*}\left(b-t_{m}\right) \chi_{\left(t_{k-1}, t_{k}\right)}\\& \qquad+\mathrm{B}^{*} \mathscr{R}^{*}(b-s) \chi_{\left(t_{m}, b\right)}\bigg) \mathscr{J}\left[\left(\alpha \mathrm{I}+\left(\Theta_{0}^{t_{m}}+\Gamma_{t_{m}}^{b}+\widetilde{\Theta}_{0}^{t_{m}}+\widetilde{\Gamma}_{t_{m}}^{b}\right) \mathscr{J}\right)^{-1}\sigma (x^{n}(\cdot))\right]\\& \qquad - \bigg(\sum_{k=1}^{m} \mathrm{B}^{*} \mathscr{R}^{*}\left(t_{k}-s\right)\\& \qquad\qquad\times \prod_{i=k+1}^{m} \mathscr{R}^{*}\left(t_{i}-t_{i-1}\right) \mathscr{R}^{*}\left(b-t_{m}\right) \chi_{\left(t_{k-1}, t_{k}\right)}+\mathrm{B}^{*} \mathscr{R}^{*}(b-s) \chi_{\left(t_{m}, b\right)}\bigg) \\
		& \qquad\times \mathscr{J}\left[\left(\alpha \mathrm{I}+\left(\Theta_{0}^{t_{m}}+\Gamma_{t_{m}}^{b}+\widetilde{\Theta}_{0}^{t_{m}}+\widetilde{\Gamma}_{t_{m}}^{b}\right) \mathscr{J}\right)^{-1}\sigma (x(\cdot))\right]\bigg\|_{\mathbb{U}} \\
		\rightarrow 0, \text { as } n \rightarrow \infty
	\end{align*}
	
	For $t \in\left[0, t_{1}\right]$, we have
	
	$$
	\begin{aligned}
		\left\|\left(G_{\alpha} x^{n}\right)(t)-\left(G_{\alpha} x\right)(t)\right\|_{\mathbb{X}} & \leq M_{\mathscr{R}} M_{\mathrm{B}} \int_{0}^{t}\left\|u_{\alpha}^{n}(s)-u_{\alpha}(s)\right\|_{\mathbb{U}} d s \\
		& +M_{\mathscr{R}} \int_{0}^{t}\left\|f\left(s, \tilde{x}^{n}_s\right)-f(s, \tilde{x}_s)\right\|_{\mathbb{X}} d s \\
		& \rightarrow 0, \text { as } n \rightarrow \infty.
	\end{aligned}
	$$
	
	For $t \in\left[t_{k}, t_{k+1}\right], k=1, \ldots, m$., we have
	
	$$
	\begin{aligned}
		&\left\|x^{n}\left(t_{k}^{+}\right)-x\left(t_{k}^{+}\right)\right\|_{\mathbb{X}}\\ &\quad =\| \sum_{i=1}^{k} \prod_{j=k}^{i+1}\left(\mathcal{I}+\mathrm{D}_{j}\right) \mathscr{R}\left(t_{j}-t_{j-1}\right)\left(\mathcal{I}+\mathrm{D}_{i}\right) \int_{t_{i-1}}^{t_{i}} \mathscr{R}\left(t_{i}-s\right) \mathrm{B}\left(u_{\alpha}^{n}(s)-u_{\alpha}(s)\right) d s \\
		&\qquad \quad +\sum_{i=1}^{k} \prod_{j=k}^{i+1}\left(\mathcal{I}+\mathrm{D}_{j}\right) \mathscr{R}\left(t_{j}-t_{j-1}\right)\left(\mathcal{I}+\mathrm{D}_{i}\right) \int_{t_{i-1}}^{t_{i}} \mathscr{R}\left(t_{i}-s\right)\left(f\left(s, \tilde{x}^{n}_s\right)-f(s, \tilde{x}_s)\right) d s \|_{\mathbb{X}}, \\
		& \quad\leq \sum_{p=1}^{k}\left(1+M_{\mathrm{D}}\right)^{p} M_{\mathscr{R}}^{p} M_{\mathrm{B}} \int_{t_{p-1}}^{t_{p}}\left\|u_{\alpha}^{n}(s)-u_{\alpha}(s)\right\|_{\mathbb{U}} d s \\
		& \qquad\quad +\sum_{p=1}^{k}\left(1+M_{\mathrm{D}}\right)^{p} M_{\mathscr{R}}^{p} \int_{t_{p-1}}^{t_{p}}\left\|f\left(s, \tilde{x}^{n}_s\right)-f(s, \tilde{x}_s)\right\|_{\mathbb{X}} d s \rightarrow 0, \text { as } n \rightarrow \infty.
	\end{aligned}
	$$
	Then, we get
	$$
	\begin{aligned}
		\left\|\left(G_{\alpha} x^{n}\right)(t)-\left(G_{\alpha} x\right)(t)\right\|_{\mathbb{X}} & \leq M_{\mathscr{R}}\left\|x^{n}\left(t_{k}^{+}\right)-x\left(t_{k}^{+}\right)\right\|_{\mathbb{X}}+M_{\mathscr{R}} M_{\mathrm{B}} \int_{t_{k}}^{t}\left\|u_{\alpha}^{n}(s)-u_{\alpha}(s)\right\|_{\mathbb{U}} d s \|_{\mathbb{X}} \\
		& +M_{\mathscr{R}} \int_{0}^{t}\left\|f\left(s, \tilde{x}^{n}_{s}\right)-f(s,\tilde{x}_{s})\right\|_{\mathbb{X}} d s \rightarrow 0, \text { as } n \rightarrow \infty .
	\end{aligned}
	$$
	
	This shows that the operator $G_{\alpha}$ is continuous. Therefore, by invoking the \textit{Schauder fixed point theorem}, we infer that $G_{\alpha}$ admits at least one fixed point in $\mathcal{B}_r$. As a result, the system described by \eqref{P} possesses a mild solution.
\end{proof}
Our next objective is to establish the \emph{approximate controllability} of the semilinear system described by \eqref{P}. Specifically, we aim to demonstrate that for any desired final state and any given level of accuracy, there exists a control function that steers the mild solution of the system arbitrarily close to this target state within the time interval $J$.
 
	\begin{thm} \label{thm3.2}
		Assume that the hypotheses \textnormal{(H1)}–\textnormal{(H3)} hold. Under these conditions, the semilinear system \eqref{P} is approximately controllable on the interval $J$.
		
	\end{thm}
	\begin{proof}
		From theorem \ref{theorem4.1}, we know that for every $\alpha >0$ and $h\in \mathbb{H}$, there exists a mild solution $x_{\alpha}\in \mathcal{PC}\left(J,\mathbb{X}\right)$ such that
		\begin{equation}
			x^{\alpha}(t)=	\begin{cases}
				\mathscr{R}(t) x(0)+\int_{0}^{t} \mathscr{R}(t-s) \big [\mathrm{B} u^{\alpha}(s)+f(s,x_{s}^{\alpha})+f_{1}^{\prime}(s)+f_{2}(s)\big] ds, \quad 0 \leq t \leq t_{1}, \\
				\mathscr{R}\left(t-t_{k}\right) x\left(t_{k}^{+}\right)+\int_{t_{k}}^{t} \mathscr{R}(t-s) \big[\mathrm{B} u^{\alpha}(s)+f(s,x_{s}^{\alpha})+f_{1}^{\prime}(s)+f_{2}(s)\big] d s, \\ \qquad\qquad t_{k}<t \leq t_{k+1},  k=1, \ldots, m,
			\end{cases}
		\end{equation}
		where
		\begin{align*}
			\begin{split}
				x\left(t_{k}^{+}\right) =&\prod_{j=k}^{1}\left(\mathrm{I}+\mathrm{D}_{j}\right) \mathscr{R}\left(t_{j}-t_{j-1}\right) \psi(0) +\sum_{i=1}^{k} \prod_{j=k}^{i+1}\left(\mathrm{I}+\mathrm{D}_{j}\right) \mathscr{R}\left(t_{j}-t_{j-1}\right)\left(\mathrm{I}+\mathrm{D}_{i}\right) \\
				&\qquad \times\left(\int_{t_{i-1}}^{t_{i}} \mathscr{R}\left(t_{i}-s\right) [\mathrm{B}u^{\alpha}(s)+f(s,x_{s}^{\alpha})] d s +\int_{t_{i-1}}^{t_{i}} \mathscr{R}(t-s) [f_{1}^{\prime}(s)+f_{2}(s)] d s\right)\\ &\qquad \quad +\sum_{i=2}^{k} \prod_{j=k}^{i}\left(\mathrm{I}+\mathrm{D}_{j}\right)\mathscr{R}\left(t_{j}-t_{j-1}\right) \mathrm{E}_{i-1} v_{i-1}+\mathrm{E}_{k} v_{k}. 
			\end{split}
		\end{align*}
		The control $u^{\alpha}(s)$ is defined as
		
		\begin{align}
			\begin{split}
				u^{\alpha}(s)=\bigg(\sum_{k=1}^{m}\mathrm{B}^{*}\mathscr{R}^{*}(t_{k}-s)\prod_{i=k+1}^{m} \mathscr{R}(t_{i}-t_{i-1})^{*}\mathscr{R}(b-t_{m})^{*}\chi(t_{k-1},t_{k})\\
				+\mathrm{B}^{*}\mathscr{R}(b-s)^{*}\chi(t_{m},b)\bigg)\widehat{\varphi}^{\alpha},\\
				v_{m}=\mathrm{E}_{m}^{*}\mathscr{R}(b-t_{m})^{*}\widehat{\varphi}^{\alpha},\quad
				v_{k}=\mathrm{E}_{k}^{*}\prod_{i=k}^{m} \mathscr{R}(t_{i}-t_{i-1})^{*}(\mathrm{I}+\mathrm{D}_{i}^{*})\mathscr{R}(b-t_{m})^{*}\widehat{\varphi}^{\alpha},
			\end{split}
		\end{align}
		with 
		\begin{align*}
			\widehat{\varphi}^{\alpha}=&\mathscr{J}\left[\left(\alpha \mathrm{I}+\left(\Theta_{0}^{t_{m}}+\Gamma_{t_{m}}^{b}+\widetilde{\Theta}_{0}^{t_{m}}+\widetilde{\Gamma}_{t_{m}}^{b}\right) \mathscr{J}\right)^{-1} \sigma\left(x(\cdot)\right)\right],
		\end{align*} 
		\begin{align*}
			\sigma\left(x(\cdot)\right)&=h-\mathscr{R}\left(b-t_{m}\right) \prod_{j=m}^{1}(\mathrm{I}+\mathrm{D}_{j}) \mathscr{R}\left(t_{j}-t_{j-1}\right) \psi(0)\\& \qquad-\int_{t_{m}}^{b} \mathscr{R}(b-s)\left[f(s,x_{s}^{\alpha})+f_{1}^{\prime}(s)+f_{2}(s)\right] d s\\
			&\qquad-\mathscr{R}\left(b-t_{m}\right)\sum_{i=1}^{m} \prod_{j=m}^{i+1}\left(\mathrm{I}+\mathrm{D}_{j}\right) \mathscr{R}\left(t_{j}-t_{j-1}\right)\left(\mathrm{I}+\mathrm{D}_{i}\right)\\& \qquad\qquad \times\int_{t_{i-1}}^{t_{i}} \mathscr{R}\left(t_{i}-s\right)\left[f(s,x_{s}^{\alpha})+f_{1}^{\prime}(s)+f_{2}(s)\right] d s.
		\end{align*}
		
		We can verify that
		
		\begin{align}\label{eq4.17}
			x^{\alpha}(b)= h-\alpha\left(\alpha \mathrm{I}+\left(\Theta_{0}^{t_{m}}+\Gamma_{t_{m}}^{b}+\widetilde{\Theta}_{0}^{t_{m}}+\widetilde{\Gamma}_{t_{m}}^{b}\right) \mathscr{J}\right)^{-1} \sigma\left(x(\cdot)\right).
		\end{align}
		By assumption (H3), we obtain
		
		\begin{align*}
			\int_{0}^{b}\left\|f(s,x^{\alpha_{i}}_s)\right\|_{\mathbb{X}} d s\leq \int_{0}^{b} \gamma(s) ds < +\infty, i\in \mathbb{N}.
		\end{align*}
		
		From above inequality, it is clear that the sequence $\left\{f(\cdot,x^{\alpha_{i}}_{(\cdot)})\right\}_{i=1}^{\infty}$ is uniformly integrable. By applying Dunford-Pettis theorem, we can find a subsequence of $\left\{f(\cdot,x^{\alpha_{i}}_{(\cdot)})\right\}_{i=1}^{\infty}$ still denoted by $\left\{f(\cdot,x^{\alpha_{i}}_{(\cdot)})\right\}_{i=1}^{\infty}$ such that 
		\begin{equation*}
			f(\cdot,x^{\alpha_{i}}_{(\cdot)})\rightharpoonup f(\cdot) in L^{1}(J,\mathbb{X}), \text{ as } \alpha_{i}\rightarrow 0^{+} (i\rightarrow \infty).
		\end{equation*}
		Further, we compute\newline
		
		\begin{align} \label{eq4.18}
			\left\|\sigma \left(x^{\alpha_{i}}(\cdot)\right)-\xi \right\|&\leq\bigg\|\int_{t_m}^{b} \mathscr{R}(b-s)\left[f(s,x^{\alpha_{i}}_s)-f(s)\right] ds\nonumber\\&\qquad-\mathscr{R}\left(b-t_{m}\right)\sum_{i=1}^{m} \prod_{j=m}^{i+1}\left(\mathrm{I}+\mathrm{D}_{j}\right) \mathscr{R}\left(t_{j}-t_{j-1}\right)\left(\mathrm{I}+\mathrm{D}_{i}\right)\nonumber\\&\qquad\quad \times\int_{t_{i-1}}^{t_{i}} \mathscr{R}\left(t_{i}-s\right)\left[f(s,x_{s}^{\alpha})-f(s)\right] d s\bigg\|_{\mathbb{X}}\nonumber\\
			&\rightarrow 0 \text{ as } \alpha_{i}\rightarrow 0^+ (i\rightarrow \infty),
		\end{align}
		where
		\begin{align*}
			\xi&=h-\mathscr{R}\left(b-t_{m}\right) \prod_{j=m}^{1}(\mathrm{I}+\mathrm{D}_{j}) \mathscr{R}\left(t_{j}-t_{j-1}\right) \psi(0)-\int_{t_{m}}^{b} \mathscr{R}(b-s)\left[f(s)+f_{1}^{\prime}(s)+f_{2}(s)\right] d s\\
			&\qquad-\mathscr{R}\left(b-t_{m}\right)\sum_{i=1}^{m} \prod_{j=m}^{i+1}\left(\mathrm{I}+\mathrm{D}_{j}\right) \mathscr{R}\left(t_{j}-t_{j-1}\right)\left(\mathrm{I}+\mathrm{D}_{i}\right)\\& \qquad\qquad \times\int_{t_{i-1}}^{t_{i}} \mathscr{R}\left(t_{i}-s\right)\left[f(s)+f_{1}^{\prime}(s)+f_{2}(s)\right] d s.
		\end{align*}
		
		The estimate \eqref{eq4.18} tends to zero by using the above weak convergence alongwith corollary 2.12 of \cite{arora2025controllability}.
		
		From the inequality \eqref{eq4.17} we have $z_{\alpha_{i}}=x^{\alpha_{i}}(b)-h$ for each $\alpha_{i}$ is a solution of the equation 
		\begin{equation*}
			\alpha_{i}z_{\alpha_{i}}+ MM^{*}\mathscr{J}[z_{\alpha_{i}}]=\alpha_{i}h_{\alpha_{i}},
		\end{equation*}
		where
		\begin{align*}
			h_{\alpha_{i}}&= -\sigma\left(x^{\alpha_{i}}(\cdot)\right)\\&= \mathscr{R}\left(b-t_{m}\right) \prod_{j=m}^{1}(\mathrm{I}+\mathrm{D}_{j}) \mathscr{R}\left(t_{j}-t_{j-1}\right) \psi(0)+\int_{t_{m}}^{b} \mathscr{R}(b-s)\left[f(s)+f_{1}^{\prime}(s)+f_{2}(s)\right] d s\\
			&\qquad+\mathscr{R}\left(b-t_{m}\right)\sum_{i=1}^{m} \prod_{j=m}^{i+1}\left(\mathrm{I}+\mathrm{D}_{j}\right) \mathscr{R}\left(t_{j}-t_{j-1}\right)\left(\mathrm{I}+\mathrm{D}_{i}\right)\\&\qquad\qquad \times\int_{t_{i-1}}^{t_{i}} \mathscr{R}\left(t_{i}-s\right)\left[f(s)+f_{1}^{\prime}(s)+f_{2}(s)\right] d s-h.
		\end{align*}
		From assumption (H1) we have the operator $MM^{*}=\left(\Theta_{0}^{t_{m}}+\Gamma_{t_{m}}^{b}+\widetilde{\Theta}_{0}^{t_{m}}+\widetilde{\Gamma}_{t_{m}}^{b}\right)$ is positive, now by aplying theorem 2.5 from \cite{mahmudov2003approximate} together with estimate \eqref{eq4.18}, we obtain
		\begin{equation*}
			\left\|x^{\alpha_{i}}(b)-h\right\|_{\mathbb{X}}=\left\| x_{\alpha_{i}}\right\|\rightarrow 0 \text{ as } \alpha_{i}\rightarrow 0^{+} (i\rightarrow\infty).
		\end{equation*}
		Therefore, the system \eqref{P} is approximately controllable.
	\end{proof}

	\section{Application}
	This section presents an example to demonstrate the applicability of the theoretical results established earlier. In particular, we study a semilinear impulsive neutral integro-delay system that emerges in the context of heat conduction in materials with fading memory. The functional framework of this example is defined within appropriate state $\mathrm{L}^{p}\left([0,\pi];\mathbb{R}\right)$ and the control space $\mathrm{L}^{2}\left([0,\pi];\mathbb{R}\right)$
		\begin{align}\label{example}
				\frac{\partial}{\partial t}\left[x(t,\zeta)+\int_{-\infty}^{t}(t-s)^{\gamma}e^{-\kappa(t-s)}x(s,\zeta) ds\right]&= \frac{\partial^{2}}{\partial\zeta^{2}}z(t,\zeta)+\eta(t,\zeta)+ \int_{-\infty}^{t}e^{-\mu(t-s)}x(s,\zeta) ds \nonumber\\  &+\int_{-\infty}^{t}\mathcal{H}(t-s)x(s,\zeta) ds,\quad \zeta\in[0,\pi], \nonumber\\
			\qquad t\in (0,b], t\ne \left\{t_1,\dots,t_m\right\}, \nonumber\\
			x(t,0)&=0=x(t,\pi), \quad t\in J, t \ne \left\{t_1,\dots,t_m\right\},\\
			\Delta x\left(t_{k},\zeta\right)&= -x\left(t_{k},\zeta\right)-v_{k}(\zeta), \quad \zeta \in (0,\pi), k=0,\dots m-1, \nonumber\\
			x\left(\theta,\zeta\right)&= \psi(\theta,\zeta),\quad \zeta\in (-\infty,0], \zeta\in [0,\pi] \nonumber,
	\end{align}
where, $\gamma \in (0,1)$ and $\kappa, \mu$ be fixed positive constants and $\mathcal{H} : [0,\infty) \rightarrow \mathbb{R}$ and $\psi : (-\infty,0] \times [0,\pi] \rightarrow \mathbb{R}$ assumed to be suitable functions. Also, let $\eta : J \times [0,\pi] \rightarrow \mathbb{R}$ be square integrable with respect to both variables.

We now reformulate the original system into its abstract version as described in \eqref{P}. To this end, we define the state space as a reflexive Banach space $\mathbb{X} = L^p([0,\pi]; \mathbb{R})$ for some $p \in [2, \infty)$ and choose the control space as $\mathbb{U} = L^2([0,\pi]; \mathbb{R})$. The phase space is taken as $\mathfrak{B} = C_0 \times L_\omega^1(\mathbb{X})$. It is known that the dual space $\mathbb{X}^* = L^{\frac{p}{p-1}}([0,\pi]; \mathbb{R})$ is uniformly convex.

We define the operator $\mathrm{A} : D(\mathrm{A}) \subset \mathbb{X} \to \mathbb{X}$ by
\begin{equation*}
	\mathrm{A}g = g'', \quad D(\mathrm{A}) = W^{2,p}([0,\pi]; \mathbb{R}) \cap W_0^{1,p}([0,\pi]; \mathbb{R}).
\end{equation*}

Let $\mathrm{N}(t)g = e^{-\mu t}g$ and $\mathrm{G}(t)g = t^\alpha e^{-\kappa t}g$ for all $g \in \mathbb{X}$. Additionally, the operator $\mathrm{A}$ admits a spectral representation given by:
\begin{equation*}
	\mathrm{A}g = -\sum_{k=1}^\infty k^2 \langle g, \tilde{a}_k \rangle\, \tilde{a}_k, \quad g \in D(\mathrm{A}),
\end{equation*}
where $\tilde{a}_k(\zeta) = \sqrt{\frac{2}{\pi}} \sin(k\zeta)$ and the inner product is defined by $\langle g, \tilde{a}_k \rangle := \int_0^\pi g(\zeta) \tilde{a}_k(\zeta)\,\mathrm{d}\zeta$.

Assume that assumptions (Cd1)–(Cd4) from \cite{arora2025controllability} hold. It then follows that the abstract linear system corresponding to the physical model \eqref{example} possesses a resolvent family $\mathcal{R}_p(\cdot)$ in $\mathbb{X}$. Define $x(t)(\zeta) := x(t,\zeta)$ and $\psi(t)(\zeta) := \psi(t,\zeta)$ for $t \in J$ and $\zeta \in [0,\pi]$.

As shown in \cite{arora2025controllability}, the functions $f : [0,b] \times \mathfrak{B} \rightarrow \mathbb{X}$ and $f_1, f_2 : [0,T] \rightarrow \mathbb{X}$ are given by
\begin{align*}
	f(t, \phi)(\zeta) &= \int_{-\infty}^0 \mathcal{H}(-s) \phi(s, \zeta)\,\mathrm{d}s, \\
	f_1(t)(\zeta) &= \int_{-\infty}^0 (t - s)^\alpha e^{-\kappa(t - s)} \psi(s, \zeta)\,\mathrm{d}s, \\
	f_2(t)(\zeta) &= \int_{-\infty}^0 e^{-\mu(t - s)} \mathrm{A} \psi(s, \zeta)\,\mathrm{d}s.
\end{align*}

These functions are well-defined, with $f$ being continuous and bounded by $\|f\|_{\mathcal{L}(\mathfrak{B}; \mathbb{X})} \leq K_f$. Furthermore, $f_1 \in C^1([0,T]; \mathbb{X})$ and $f_2 \in C([0,T]; \mathbb{X})$.

The control operator $\mathrm{B} : \mathbb{U} \rightarrow \mathbb{X}$ is defined by
\begin{equation*}
	\mathrm{B}u(t)(\zeta) := \eta(t,\zeta) = \int_0^\pi K(\omega, \zeta) u(t)(\omega)\,\mathrm{d}\omega, \quad t \in J, \ \zeta \in [0,\pi],
\end{equation*}
where $K \in C([0,\pi] \times [0,\pi]; \mathbb{R})$ is symmetric: $K(\omega,\zeta) = K(\zeta,\omega)$. It is known that under these assumptions, $\mathrm{B}$ is bounded (see \cite{wu2012theory}). One can choose a specific kernel—for instance, $K(\zeta, \omega) = \min\{\zeta, \omega\}$—to ensure that $\mathrm{B}$ is injective.

For the impulsive effect, we consider the condition:
\[
\Delta x(t_k, \zeta) = -x(t_k, \zeta) - v_k(\zeta),
\]
which corresponds to choosing the operators $\mathrm{D}_k = \mathrm{E}_k = -\mathrm{I}$.

With all these elements in place, the system \eqref{example} can be equivalently cast into the abstract form \eqref{P}, which satisfies the assumptions required for controllability analysis.

To verify the approximate controllability of the linearized form of system \eqref{P}, suppose for contradiction that there exists $w^* \in \mathbb{X}^*$ such that $\mathrm{B}^* \mathcal{R}_p(T - t)^* w^* = 0$ for all $t \in [0,T]$. Then, one obtains:
\[
\mathrm{B}^* \mathcal{R}_p(T - t)^* w^* = 0 \Rightarrow \mathcal{R}_p(T - t)^* w^* = 0 \Rightarrow w^* = 0,
\]
which leads to a contradiction unless $w^* = 0$. Hence, the corresponding linear system is approximately controllable.

Applying Theorem~\ref{thm3.2}, we conclude that the semilinear system \eqref{P}, and therefore system \eqref{example}, is approximately controllable.

	\section{Concluding remarks and future scope}
	We discussed properties of an impulsive resolvent operator in a reflexive Banach space with a uniform convex dual. First, we write the mild solution with the resolvent family $\mathscr{R}$ as defined in definition \refeq{def2.1}, then we move to define and establish the properties of the impulsive resolvent operator. We checked for the approximate controllability of the linear system \eqref{LP} and then, by defining a feedback control, we proved the existence of a mild solution with this feedback control by using the Schauder fixed-point theorem. Subsequently, we proved the approximate controllability of the problem considered \eqref{P}. Finally, we give an example of the heat conduction phenomenon with fading memory to verify all the results achieved. This study can be generalised to study the approximate controllability and finite approximate controllability of evolution equations with singular memory in a more general Banach space.
	
	\section{Acknowledgement}
	The first author, Garima Gupta, gratefully acknowledges the financial support received from the Ministry of Education, Government of India, in the form of an MHRD fellowship during the course of this research. The authors also expresses sincere thanks to peers for their valuable suggestions and encouragement throughout this work.
	
	\section{Author contribution declaration}
	\textbf{Garima Gupta:} Problem formulation, writing original draft, methodology, conceptualization.
	\textbf{Jaydev Dabas:} Problem formulation, reviewing, editing, conceptualization. 
	
	\section{Declaration of competing interest}
	The author declare that they have no conflict of interests.

	\bibliographystyle{elsarticle-num}
	\bibliography{references}

\begin{thebibliography}{10}
\expandafter\ifx\csname url\endcsname\relax
  \def\url#1{\texttt{#1}}\fi
\expandafter\ifx\csname urlprefix\endcsname\relax\def\urlprefix{URL }\fi
\expandafter\ifx\csname href\endcsname\relax
  \def\href#1#2{#2} \def\path#1{#1}\fi

\bibitem{gurtin1968general}
M.~E. Gurtin, A.~C. Pipkin, A general theory of heat conduction with finite
  wave speeds, Archive for Rational Mechanics and Analysis 31 (1968) 113--126.

\bibitem{nunziato1971heat}
J.~W. Nunziato, On heat conduction in materials with memory, Quarterly of
  Applied Mathematics 29~(2) (1971) 187--204.

\bibitem{lunardi1990linear}
A.~Lunardi, On the linear heat equation with fading memory, SIAM Journal on
  Mathematical analysis 21~(5) (1990) 1213--1224.

\bibitem{wu2012theory}
J.~Wu, Theory and applications of partial functional differential equations,
  Vol. 119, Springer Science \& Business Media, 2012.

\bibitem{dos2011existence}
J.~P.~C. Dos~Santos, H.~Henr{\'\i}quez, E.~Hern{\'a}ndez, Existence results for
  neutral integro-differential equations with unbounded delay, The Journal of
  Integral Equations and Applications (2011) 289--330.

\bibitem{arora2025controllability}
S.~Arora, A.~Nandakumaran, Controllability problems of a neutral
  integro-differential equation with memory, Nonlinear Analysis: Real World
  Applications 84 (2025) 104317.

\bibitem{lakshmikantham1989theory}
V.~Lakshmikantham, P.~S. Simeonov, et~al., Theory of impulsive differential
  equations, Vol.~6, World scientific, 1989.

\bibitem{george2000note}
R.~George, A.~Nandakumaran, A.~Arapostathis, A note on controllability of
  impulsive systems, Journal of Mathematical Analysis and Applications 241~(2)
  (2000) 276--283.

\bibitem{arora2020approximate}
S.~Arora, S.~Singh, J.~Dabas, M.~T. Mohan, Approximate controllability of
  semilinear impulsive functional differential systems with non-local
  conditions, IMA Journal of Mathematical Control and Information 37~(4) (2020)
  1070--1088.

\bibitem{arora2021approximate}
S.~Arora, M.~T. Mohan, J.~Dabas, Approximate controllability of
  non-instantaneous impulsive fractional evolution equations of order
  $\alpha\in(1,2)$ with state-dependent delay in banach spaces, arXiv preprint
  arXiv:2106.15122.

\bibitem{mahmudov2024study}
N.~I. Mahmudov, A study on approximate controllability of linear impulsive
  equations in hilbert spaces, Quaestiones Mathematicae (2024) 1--16.

\bibitem{asadzade2024approximate}
J.~A. Asadzade, N.~I. Mahmudov, Approximate controllability of impulsive
  semilinear evolution equations in hilbert spaces, arXiv preprint
  arXiv:2411.02766.

\bibitem{asadzade2025remarks}
J.~A. Asadzade, N.~I. Mahmudov, Remarks on finite-approximate controllability
  of impulsive evolution systems via resolvent-like operator in hilbert spaces,
  arXiv preprint arXiv:2501.02995.

\bibitem{gupta2024existence}
G.~Gupta, J.~Dabas, The existence and controllability of nonautonomous system
  influenced by impulses on both state and control, arXiv preprint
  arXiv:2412.01355.

\bibitem{pisier1975martingales}
G.~Pisier, Martingales with values in uniformly convex spaces, Israel Journal
  of Mathematics 20 (1975) 326--350.

\bibitem{barbu2012convexity}
V.~Barbu, T.~Precupanu, Convexity and optimization in Banach spaces, Springer
  Science \& Business Media, 2012.

\bibitem{hino2006functional}
Y.~Hino, S.~Murakami, T.~Naito, Functional differential equations with infinite
  delay, Springer, 2006.

\bibitem{mahmudov2003approximate}
N.~I. Mahmudov, Approximate controllability of semilinear deterministic and
  stochastic evolution equations in abstract spaces, SIAM journal on control
  and optimization 42~(5) (2003) 1604--1622.

\bibitem{wei2006nonlinear}
W.~Wei, X.~Xiang, Y.~Peng, Nonlinear impulsive integro-differential equations
  of mixed type and optimal controls, Optimization 55~(1-2) (2006) 141--156.

\end{thebibliography}
	
\end{document}